\documentclass{article}
\usepackage{graphicx} % Required for inserting images
\usepackage{amsthm}
\usepackage{amsfonts, amsmath, amssymb, amsthm}
\usepackage{hyperref}
\usepackage{cleveref}
\usepackage[utf8x]{inputenc}
\usepackage{url}

\newcommand{\BSig}{\mathsf{B}\Sigma^0}

\newcommand{\RCA}[0]{\mathsf{RCA}}

\newcommand{\ADS}[0]{\mathsf{ADS}}
\newcommand{\EM}[0]{\mathsf{EM}}

\newcommand{\RT}[0]{\mathsf{RT}}

\newcommand{\trRT}[0]{\mathsf{trRT}}

\newcommand{\fEM}[0]{\mathsf{fEM}}

\newcommand{\NN}[0]{\mathbb{N}}

\newcommand{\card}{\operatorname{card}}
\newcommand{\dom}{\operatorname{dom}}
\newcommand{\finsub}{\subseteq_{\mathtt{fin}}}

\def\qt#1{``#1''}%

\newcommand{\psn}{\mathrm{psn}}

\title{Partition theorems for Ketonen-Solovay largeness: Hardy-like version}
\date{\today}

\newtheorem*{statement}{Statement}
\newtheorem{theorem}{Theorem}
\numberwithin{theorem}{section}
\newtheorem{maintheorem}[theorem]{Main Theorem}
\newtheorem{lemma}[theorem]{Lemma}

\newtheorem{proposition}[theorem]{Proposition}
\newtheorem{remark}[theorem]{Remark}
\newtheorem{definition}[theorem]{Definition}
\newtheorem{corollary}[theorem]{Corollary}
\newtheorem{example}[theorem]{Example}

\makeatletter
\newtheorem*{rep@theorem}{\rep@title}
\newcommand{\newreptheorem}[2]{%
\newenvironment{rep#1}[1]{%
 \def\rep@title{#2 \ref{##1}}%
 \begin{rep@theorem}}%
 {\end{rep@theorem}}}
\makeatother

\newreptheorem{theorem}{Theorem}
\newreptheorem{maintheorem}{Main Theorem}

\usepackage{xcolor}

\usepackage{authblk}
\DeclareSymbolFont{bbold}{U}{bbold}{m}{n}
\DeclareMathSymbol{\bbomega}{\mathord}{bbold}{"7F}

\author{Quentin Le Houérou \and Ludovic Patey}

\begin{document}

\maketitle

\begin{abstract}
We develop the framework of $\alpha$-largeness introduced by Ketonen and Solovay, by proving a partition theorem\footnote{This is a preliminary version of the article, based on the Hardy-like hierarchy to prove \Cref{thm:epsilon0-pigeonhole}. An improved version based on a variant of the $\Rightarrow_x$ relation is in preparation.} for $\alpha$-large sets with $\alpha < \epsilon_0$ which generalizes theorems from Ketonen and Solovay and from Bigorajska and Kotlarski. We also prove that for every $\omega^{nk+3}$-large set~$X$ with $\min X \geq 18$, every coloring $f : [X]^2 \to k$ admits an $\omega^n$-large $f$-homogeneous subset. This bound is tight, up to an additive constant.
\end{abstract}

\section{Introduction}

The celebrated Paris-Harrington theorem states that Peano arithmetics does not prove the existence of $\omega$-large homogeneous sets for arbitrary instances of Ramsey's theorem. Here, a finite set~$X$ is $\omega$-large if $\card X > \min X$. Ketonen and Solovay~\cite{ketonen1981rapidly} defined a notion of $\alpha$-largeness for $\alpha < \epsilon_0$ to quantify the size of finite sets over weak theories, and related it to the Wainer hierarchy of fast-growing functions~\cite{wainer1970classification}. More precisely, any fundamental sequence system $(\{\alpha\}(n))_{n \in \NN}$ for the ordinals~$\alpha < \epsilon_0$ induces a notion of largeness defined as follows: a finite set of integers $\{ x_0 < \dots < x_{k-1} \}$ is \emph{$\alpha$-large} if, letting $\alpha_0 = \alpha$ and $\alpha_{s+1} = \{\alpha_s\}(x_s)$, the ordinal $\alpha_{k-1}$ equals~0.

Since then, variants of Ketonen and Solovay's notion of largeness have been extensively studied, both from a combinatorial~\cite{bigorajska2006some,bigorajska1999partition,kotlarski2007more,bigorajska2002combinatorics,smet2010partitioning} and a proof-theoretic perspective~\cite{ratajczyk1988combinatorial,kotlarski2019model,patey2018proof,kolodziejczyk2023ramsey,pelupessy2016largeness,towsner2024erdos}, to prove partial conservation theorems over subsystems of second-order arithmetic for $\forall \Sigma^0_2$-sentences and non-speedup theorems, where a $\forall \Sigma^0_n$-formula is of the form $\forall X \varphi(X)$ where $\varphi(X)$ is $\Sigma^0_n$. More recently, parameterized versions of largeness have been defined to prove partial conservation theorems over $\RCA_0 + \BSig_2$ for $\forall \Sigma^0_3$-sentences~\cite{houerou2023conservation,houerou2026conservationramsey}.

The framework of $\alpha$-largeness is however relatively scattered: Ketonen and Solovay~\cite{ketonen1981rapidly} first proved some general structural theorems about $\alpha$-largeness for $\alpha < \epsilon_0$. Then Bigorajska and Kotlarski~\cite{bigorajska2002combinatorics,kotlarski2019model,bigorajska2006some,bigorajska1999partition} published a series of articles on a variant of $\alpha$-largeness, say $\alpha$-largeness${}^*$, introduced by Ratajczyk~\cite{ratajczyk1988combinatorial} based on the Hardy hierarchy of fast-growing functions~\cite{hardy1904theorem}. They proved in particular a general pigeonhole theorem for $\alpha$-largeness${}^*$ with $\alpha < \epsilon_0$. A framework for proving lower bounds for $\alpha$-largeness${}^*$ was developed by multiple authors~\cite{bigorajska2006some,kotlarski2007more}. Both notions of largeness are related (in particular, any $\alpha$-large set is $\alpha$-large${}^*$ for $\alpha \geq \omega$), but the translation of structural theorems from one notion to the author yields sub-optimal statements. Independently, the study of partial conservation theorems over~$\mathsf{I}\Sigma_1$ motivated the development of the framework of $\alpha$-largeness for $\alpha < \omega^\omega$~\cite{patey2018proof,kolodziejczyk2023ramsey,pelupessy2016largeness,towsner2024erdos}. In particular, Ko{\l}odziejczyk and Yokoyama~\cite{kolo2020some} proved that for every $\omega^{144(n+1)}$-large set~$F$ and every 2-coloring $f : [F]^2 \to 2$, there exists an $\omega^n$-large $f$-homogeneous set. For this, they proved multiple structural theorems in the restricted setting of~$\alpha < \omega^\omega$.

In this article, we pursue the general study of Ketonen and Solovay's notion of largeness, by proving the following general partition theorem, where $\oplus$ denotes the natural (or Hessenberg) sum over ordinals, and a set~$F$ is \emph{at most $\alpha$-large} if $F \setminus \{\max F\}$ is $\alpha$-small.

\begin{maintheorem}\label[maintheorem]{thm:epsilon0-pigeonhole}
    Let $\beta, \gamma < \epsilon_0$ be two ordinals, and $B,C \subseteq \NN \setminus \{0\}$ be at most $\beta$-large and $\gamma$-large respectively. Then $B \cup C$ is at most $(\beta \oplus \gamma)$-large.
\end{maintheorem}

This theorem has several consequences, including a generalization of Ketonen and Solovay~\cite[Lemma 4.6]{ketonen1981rapidly} and pigeonhole generalizations of Bigorajska and Kotlarski~\cite[Theorem 12]{bigorajska1999partition} for $\alpha$-largeness. For this, we introduce a Hardy-like hierarchy of fast-growing functions, and adapt and generalize the techniques of Bigorajska and Kotlarski.

Our second contribution is a tight upper bound (up to additive constant) of the closure of $\alpha$-largeness under Ramsey's theorem for pairs, for $\alpha < \omega^\omega$. We say that a set~$F$ is \emph{$\RT^2_k$-$\alpha$-large} if for every $k$-coloring of the pairs $[F]^2$, there is an $\alpha$-large $f$-homogeneous subset. We prove the following theorem:

\begin{maintheorem}\label[maintheorem]{thm:rt22-tight-bound}
Let $n, k \geq 1$. If $X \finsub \NN$ is $\omega^{kn+3}$-large and $\min X \geq 17$, then it is $\RT^2_k$-$\omega^n$-large. 
\end{maintheorem}

Note that in the case $k = 2$, we obtain $\omega^{2n+3}$-largeness, improving over the $\omega^{144(n+1)}$-largeness upper bound of Ko{\l}odziejczyk and Yokoyama~\cite{kolo2020some}. This bound is tight, in the sense that Kotlarski et al.~\cite[Theorem 5.4]{kotlarski2007more} proved that if a set is $\RT^2_k$-$\omega^n$-large${}^*$, then it is $\omega^{kn}$-large${}^*$. Translating in the $\alpha$-largeness framework, this shows that every $\RT^2_k$-$\omega^n$-large is $\omega^{kn-1}$-large. 

The proof of \Cref{thm:rt22-tight-bound} goes through the computation of upper bounds of closure for two intermediate combinatorial theorems, namely, the restriction of Ramsey's theorem for pairs to transitive colorings ($\trRT^2_k$) and a generalization of the Erd\H{o}s-Moser theorem to $k$-colorings ($\fEM$) due to Towsner and Yokoyama~\cite{towsner2024erdos}. We prove in particular that $\trRT^2_k$ is solely responsible for the lower bound of $\RT^2_k$-$\omega^n$-largeness, in that every $\omega^{n+3}$-large set~$X$ with $\min X \geq 7$ is $\fEM$-$\omega^n$-large (\Cref{cor:em-n+3-large}). This later result improves the $\omega^{18n}$-largeness upper bound of Towsner and Yokoyama for $\fEM$-$\omega^n$-largeness.

The remainder of this article is divided into two parts, as follows: In \Cref{sec:largeness-epsilon0}, we define and study Ketonen and Solovay's notion of $\alpha$-largeness for $\alpha < \epsilon_0$ and prove \Cref{thm:epsilon0-pigeonhole} and its consequences. Then, in \Cref{sec:better-bounds-regular-largeness}, we restrict ourselves to $\alpha$-largeness for $\alpha < \epsilon_0$ and study the closure of $\alpha$-largeness under multiple combinatorial statements, including Ramsey's theorem for pairs. The section culminates with the proof of \Cref{thm:rt22-tight-bound}.

\section{Largeness below $\epsilon_0$}\label[section]{sec:largeness-epsilon0}

Every ordinal $\alpha < \epsilon_0$ admits a unique \emph{Cantor normal form} 
$$\alpha = \omega^{\alpha_s} \cdot a_s + \cdots + \omega^{\alpha_0} \cdot a_0$$
for some~$\alpha_0 < \alpha_1 < \cdots < \alpha_s < \alpha$ and $a_0, \dots, a_s \in \NN \setminus \{0\}$. Given two ordinals $\alpha, \beta < \epsilon_0$, we write $\beta \gg \alpha$ if either $\alpha$ or $\beta$ equals 0, or if the smallest exponent of~$\beta$ is at least the largest exponent of~$\alpha$ in their Cantor normal form.
If $\alpha \neq 0$, we shall also use its \emph{short Cantor normal form} $\alpha = \beta + \omega^\gamma$
with $\beta \gg \omega^\gamma$.

\begin{definition}
Given a non-zero ordinal $\alpha < \epsilon_0$ with short Cantor normal form $\alpha = \beta + \omega^\gamma$, and $x \in \NN$, let 
$$\{\alpha\}(x) = \begin{cases}
        \beta & \text{ if } \gamma = 0,\\
        \beta + \omega^{\gamma-1} \cdot x & \text{ if } \gamma \text{ is successor },\\
        \beta + \omega^{\{\gamma\}(x)} & \text{ if } \gamma \text{ is limit.}
    \end{cases}
$$
Finally, let $\{0\}(x) = 0$.
\end{definition}

The sequence $(\{\alpha\}(x))_{x \in \NN}$ is called the (canonical) \emph{fundamental sequence} of $\alpha$. When $\alpha$ is a limit ordinal, then it is a distinguished increasing sequence of ordinals that is cofinal in $\alpha$.

The $\gg$ relation admits two desirable properties. First, if $\beta \gg \alpha$, then the sum $\alpha + \beta$ coincides with the natural sum (or Hessenberg sum) $\alpha \oplus \beta$ corresponding to the component-wise sum of the multiplicative factors. Second, if $\beta \gg \alpha$ and $\alpha \neq 0$, then the operation $\{\beta + \alpha\}(x)$ can be reduced to a computation of $\{\alpha\}(x)$:

\begin{lemma}[Ketonen and Solovay~\cite{ketonen1981rapidly}]\label[lemma]{lem:mesh-evaluation}
Let $\alpha, \beta < \epsilon_0$ be such that $\beta \gg \alpha$ and $\alpha \neq 0$. Then $\{\beta + \alpha\}(x) = \beta + \{\alpha\}(x)$.
\end{lemma}

In general, fundamental sequences are not compatible with ordinal inequality, in the sense that there exist some ordinals $\beta > \alpha$ and some~$x \in \NN$ such that $\{\beta\}(x) < \{\alpha\}(x)$. For instance, if $\beta = \omega$ and $\alpha = 10$, then $\{\omega\}(1) = 1 < \{10\}(1) = 9$. However, this is the case if $x$ is larger than every coefficient appearing in the Cantor normal form of~$\alpha$. This motivates the definition of the pseudo-norm:

\begin{definition}
The \emph{pseudo-norm} $\psn(\alpha)$ of~$\alpha$ of Cantor normal form $\omega^{\alpha_s} \cdot a_s + \cdots + \omega^{\alpha_0} \cdot a_0$ is defined as 
$$\psn(\alpha) = \max(\psn(\alpha_0), \dots, \psn(\alpha_s), a_0, \dots, a_s)$$
\end{definition}

The following lemma is an immediate consequence of Bigorajska and Kotlarski~\cite[Lemma 2.1]{bigorajska2006some}.

\begin{lemma}[Bigorajska and Kotlarski~\cite{bigorajska2006some}]\label[lemma]{lem:standard-sequence-vs-psn}
    Let $\alpha < \beta < \epsilon_0$ and $x \in \NN \setminus \{0\}$ such that $\psn(\alpha) < x$. Then $\{\beta\}(x) \geq \alpha$ (with equality if only if $\beta = \alpha + 1$).
\end{lemma}

% \begin{proof}
% Proceed by induction on $\beta$. Let $\beta < \epsilon_0$ and assume the property to hold for every $\beta' < \beta$. Let $\alpha < \beta$ and $a \in \NN$ such that $\psn(\alpha) < a$.

% Write $\beta = \omega^{\beta_n}b_n + \dots + \omega^{\beta_0}b_0$ its Cantor normal form. Since $\beta > \alpha$ we have $\beta \neq 0$, hence, $\{\beta\}(a) = \omega^{\beta_n}b_n + \dots + \omega^{\beta_0}(b_0 - 1) + \{\omega^{\beta_0}\}(a)$. 

% \textbf{Case 1: $\beta_0 = 0$.} In that case, $\beta$ is successor and we have $\{\beta\}(a) \geq \alpha$. \\

% If $\alpha < \omega^{\beta_n}b_n + \dots + \omega^{\beta_0}(b_0 - 1)$, then we are done, so assume otherwise. Since $\alpha < \beta$, write $\alpha = \omega^{\beta_n}b_n + \dots + \omega^{\beta_0}(b_0 - 1) + \omega^{\alpha_k}a_k + \dots + \omega^{\alpha_0}a_0$ the Cantor normal form of $\alpha$. \\

% \textbf{Case 2: $\beta_0 = \beta_0' + 1$ for some $\beta_0' < \epsilon_0$.} In this case, $\beta_0' \geq \alpha_k$ and $\{\omega^{\beta_0}\}(a) = \omega^{\beta_0'}a > \omega^{\alpha_k} a_k + \dots + \omega^{\alpha_0} a_0$ as $a > \psn(\alpha) \geq a_{k}$, thus $\{\beta\}(a) > \alpha$.

% \textbf{Case 3: $\beta_0$ limit.} In that case, we can apply the inductive hypothesis on $\alpha_k < \beta_0 < \epsilon_0$ as $\psn(\alpha_k) \leq \psn(\alpha) < a$, thus $\{\beta_0\}(a) > \alpha_k$ (no equality since $\beta_0$ is successor), and $\{\omega^{\beta_0}\}(a) = \omega^{\{\beta_0\}(a)} \geq \omega^{\alpha_k + 1} > \omega^{\alpha_k} a_k + \dots + \omega^{\alpha_0} a_0$, thus $\{\beta\}(a) > \alpha$.
% \end{proof}

The following easy lemma relates $\psn(\{\alpha\}(n))$ to $\psn(\alpha)$ for any~$n \in \NN$.

\begin{lemma}\label[lemma]{lem:one-step-psn}
	Let $\alpha < \epsilon_0$ and let $n \in \NN$. Then $\psn(\alpha) - 1 \leq \psn(\{\alpha\}(n)) \leq \max \{\psn(\alpha), n\}$
\end{lemma}

\begin{proof}
The result is clear for $\alpha = 0$, then, if $\alpha \neq 0$, the operation $\{\alpha\}(n)$ can only add $n's$ and decrease coefficients by $1$ in the (recursive) Cantor normal form of $\alpha$.
%Proceed by induction on $\alpha$. The result holds for $\alpha = 0$. Let $\alpha > 0$ and suppose that the result holds for every $\alpha' < \alpha$.

%Write $\alpha = \omega^{\alpha_n}a_n + \dots + \omega^{\alpha_0}a_0$ its Cantor normal form. We have $\psn(\alpha) = max(\psn(\alpha_n), \dots, \psn(\alpha_0), a_n, \dots, a_0)$. There are two cases:

%\textbf{Case 1: $\alpha_0 = \alpha_0' + 1$ for some $\alpha_0 < \epsilon_0$.} In that case, $\{\alpha\}(a) = \omega^{\alpha_n}a_n + \dots + \omega^{\alpha_0}(a_0 - 1) + \omega^{\alpha_0'}a$. 

%There are four subcases depending on wether $a_0 = 1$ or not and wether $\psn(\alpha_0')$ is equal to $\psn(\alpha_0)$ or to $\psn(\alpha_0) - 1$. In every such case we have $\min \{\psn(\alpha) - 1, a\} \leq \max \{\psn(\alpha) - 1, a\} \leq \psn(\{\alpha\}(a)) \leq \max \{\psn(\alpha), a\}$

%\textbf{Case 2: $\alpha_0$ limit.} In that case, $\{\alpha\}(a) = \omega^{\alpha_n}a_n + \dots + \omega^{\alpha_0}(a_0 - 1) + \omega^{\{\alpha_0\}(a)}$. By the inductive hypothesis, we have $min \{\psn(\alpha_0) - 1, a\} \leq \psn(\{\alpha_0\}(a)) \leq \max \{\psn(\alpha_0), a\}$. Thus $\min \{\psn(\alpha) - 1,a\} \leq \psn(\{\alpha\}(a)) \leq \max \{\psn(\alpha), a\}$.
\end{proof}

The notion of fundamental sequence generalizes to finite sets as follows:

\begin{definition}[Ketonen and Solovay]
Given an ordinal~$\alpha < \epsilon_0$ and a finite set~$X = \{x_0 < \cdots < x_s \}$, let
$$
\{\alpha\}(X) = \{\cdots \{\{\alpha\}(x_0)\}(x_1) \cdots\}(x_s)
$$
A set~$X$ is \emph{$\alpha$-large} if $\{\alpha\}(X) = 0$, otherwise, it is \emph{$\alpha$-small}.
\end{definition}

Accordingly, a finite set~$X$ is \emph{exactly $\alpha$-large} if it is $\alpha$-large and $X \setminus \{\max X\}$ is $\alpha$-small. A finite set~$X$ is \emph{at most $\alpha$-large} if $X \setminus \{\max X\}$ is $\alpha$-small. By convention, we consider that $\emptyset$ is at most 0-large.

\begin{example}
Any finite set~$X \subseteq \NN$ is $\card X$-large. A finite set~$X$ is $\omega$-large if $\card X > \min X$.
A finite set~$X$ is $\omega^\beta \cdot n$-large iff $X$ contains $n$ $\omega^\beta$-large subsets $X_0 < \dots < X_{n-1} \subseteq X$, where $X < Y$ means $\forall x \in X \forall y \in Y x < y$.
\end{example}

Note that a set~$X$ might be $\alpha$-large, $\alpha+1$-small, but not at most $\alpha$-large. For instance, $\{3, 5, 6, 7, 8\}$ is $\omega$-large, not $\omega+1$-large, and not at most $\omega$-large.

\subsection{Hardy-like hierarchy}

Ratajczyk~\cite[Section 5]{ratajczyk1988combinatorial} defined a notion of largeness based on the Hardy hierarchy of functions~\cite{hardy1904theorem}. As noted by Ratajczyk, this notion is slightly different from Ketonen and Solovay's $\alpha$-largeness, and was extensively studied by Bigorajska and Kotlarski~\cite{bigorajska1999partition,bigorajska2006some,kotlarski2019model}. We now define a Hardy-like hierarchy whose induced notion of largeness corresponds exactly to Ketonen and Solovay's definition.

\begin{definition}
    Fix an increasing function $h : A \to \NN$ for some $A \subseteq \NN$ with $h(a) > a$ for every $a \in A$. We define inductively a sequence of partial functions $h_{\alpha}$ for every $\alpha < \epsilon_0$ as follows: for $x \in \NN$ and $\alpha > 0$, let $h_0(x) = x$ and let $h_{\alpha}(x) = h_{\{\alpha\}(x)}(h(x))$ if $h(x) \in \dom(h_{\{\alpha\}(x)})$ and $x \in \dom(h)$, otherwise let $h_{\alpha}(x)$ undefined.
    We write $h_{\alpha}(x) \downarrow$ is $x \in \dom h_{\alpha}$ and $h_{\alpha}(x) \uparrow$ otherwise.
\end{definition}

Note that for every~$\alpha < \epsilon_0$ and $x \in \dom h_\alpha$, there is some~$n \in \NN$ such that $h_\alpha(x) = h^{(n)}(x)$. However, this $n$ depends on~$x$: if $\alpha \geq \omega$, there is no $n$ such that $h_\alpha = h^{(n)}$.

\begin{remark}
In the definition of the original Hardy hierarchy~\cite{hardy1904theorem} based on~$h$, if $\alpha$ is limit, then $h_\alpha(x) = h_{\{\alpha\}(x)}(x)$, while in our Hardy-like hierarchy, we define $h_\alpha(x) = h_{\{\alpha\}(x)}(h(x))$. This variation reflects the main difference between the notion of largeness from Ratajczyk~\cite[Section 5]{ratajczyk1988combinatorial} and the one of Ketonen and Solovay~\cite{ketonen1981rapidly}. 
In the former case, a set~$A$ is $\alpha$-large for a limit ordinal~$\alpha$ if it is $\{\alpha\}(\min A)$-large, while in the latter case, a set $A$ is $\alpha$-large if $A \setminus \{\min A\}$ is $\{\alpha\}(\min A)$-large. Some further differences exist: in the finite case, a set $A$ is $n$-large in the sense of Ratajczyk if $\card A \geq n+1$ while in Ketonen and Solovay's sense, a set $A$ is $n$-large iff $\card A \geq n$. As a consequence, under both definitions, a set $A$ is $\omega$-large iff $\card A > \min A$.
\end{remark}

\begin{lemma}\label[lemma]{lem:psn-nb-step}
    If $h_{\alpha}(x) \downarrow$ for some $\alpha < \epsilon_0$ and $x \in \NN \setminus \{0\}$, then $h_{\alpha}(x) \geq \psn(\alpha) + x$. 
    
    Furthermore, the inequality is strict when $\alpha \geq \omega$.
\end{lemma}
\begin{proof}
By induction on $\alpha$. For $\alpha = 0$, let $x \in \NN \setminus \{0\}$ be such that $h_0(x) \downarrow$, then $h_0(x) = x = \psn(0) + x$. Let $\alpha < \epsilon_0$ and assume the property to hold for every $\alpha' < \alpha$. Let $x \in \NN$ be such that $h_{\alpha}(x) \downarrow$, then $h_{\alpha}(x) = h_{\{\alpha\}(x)}(h(x))$. By the inductive hypothesis, we have $h_{\{\alpha\}(x)}(h(x)) \geq \psn(\{\alpha\}(x)) + h(x)$, but, as $h(x) \geq x + 1$ and as $\psn(\{\alpha\}(x)) \geq \psn(\alpha) - 1$ by \Cref{lem:one-step-psn}, we have $h_{\{\alpha\}(x)}(h(x)) \geq \psn(\alpha) + x$. This concludes our induction. \\

For strictness: when $\alpha = \omega$, we have $\{\alpha\}(x) = x$, so $\psn(\{\alpha\}(x)) = \psn(x) = x > \psn(\alpha) - 1 = 0$. Thus, the inequality $\psn(\{\alpha\}(x)) \geq \psn(\alpha) - 1$ used above is strict, and hence so is $h_\alpha(x) \geq \psn(\alpha)+x$. The remaining cases $\alpha > \omega$ follow by the same induction as above, propagating strictness upward.
\end{proof}

Fix an increasing function $h : A \to \NN$ for some~$A \subseteq \NN$, with $h(x) > x$ for every~$x \in A$. The Hardy-like iterations of~$h$ yield faster-growing functions on a restricted domain, that is, if $\alpha > 0$, then $\dom h_\alpha \subseteq \dom h$ and $h_\alpha$ dominates $h$ on $\dom h_\alpha$ (see \Cref{lem:h-hierarchy}(4)).

On the other hand, given two ordinals $\alpha \leq \beta < \epsilon_0$ and some~$x \in \dom h_\beta$, it is not true in general that $x \in \dom h_\alpha$, nor that if $x \in \dom h_\beta \cap \dom h_\alpha$, then $h_\beta(x) \geq h_\alpha(x)$. However, if $x$ is sufficiently large -- $x > \psn(\alpha)$ suffices --, then this is the case (see \Cref{lem:standard-sequence-vs-psn-2,lem:h-hierarchy}). However, the $x > \psn(\alpha)$ hypothesis is too strong for our purpose, so we will use a more general relation $\beta \Rightarrow_x \alpha$ which depends on both $\beta$ and $\alpha$, and such that if $\beta \Rightarrow_x \alpha$, then $\beta \geq \alpha$ and $h_\beta(x) \geq h_\alpha(x)$.

\begin{definition}[Ketonen and Solovay]
Given two ordinals $\alpha, \beta < \epsilon_0$ and $x \geq 0$,
let $\beta \Rightarrow_x \alpha$ if either $\beta = \alpha$, or $\{\beta\}(x) \Rightarrow_x \alpha$.
\end{definition}

Bigorajska and Kotlarski~\cite[Lemma 2(viii)]{bigorajska1999partition} proved that the definition is unchanged if one replaces $\{\beta\}(x)$ by $\{\beta\}(y)$ for some~$y \leq x$. As one expects from the previous discussion, the relation $\beta \Rightarrow_x \alpha$ coincides with the inequality $\beta \geq \alpha$ whenever $x > \psn(\alpha)$. The whole interest of this relation is therefore in the case $x \leq \psn(\alpha)$.

The following lemma strengthens Ketonen and Solovay~\cite[Proposition 2.8.1]{ketonen1981rapidly} and is an immediate consequence of Bigorajska and Kotlarski~\cite[Lemma 2.1]{bigorajska2006some}. It appears in this form in the posthumous book of Kotlarski~\cite[Lemma~1.4.3]{kotlarski2019model}.

%Intuitively, the lemma says that when $\psn(\alpha) < x$, then the relation $\beta \Rightarrow_x \alpha$ is nothing but the relation $\beta \geq \alpha$. The whole interest of the $\Rightarrow_x$ relation is therefore to apply \Cref{prop:h-hierarchy}(2) when $x \leq \psn(\alpha)$, which will be the case in the proof of \Cref{prop:pigeonhole-one-step}.

\begin{lemma}[Kotlarski~\cite{kotlarski2019model}]\label[lemma]{lem:standard-sequence-vs-psn-2}
Let $\alpha \leq \beta < \epsilon_0$ and $x \in \NN$ such that $\psn(\alpha) < x$, then $\beta \Rightarrow_x \alpha$.
\end{lemma}

Ketonen and Solovay~\cite{ketonen1981rapidly} proved many properties about the relation $\Rightarrow_x$. The following lemma compiles a list of basic facts that shall be useful to prove our main theorem.

\begin{lemma}[Ketonen and Solovay~\cite{ketonen1981rapidly}]\label[lemma]{lem:basic-arrow}
Let $\alpha, \beta, \lambda < \epsilon_0$ and $x \geq 0$.
\begin{itemize}
    \item[(1)] If $\lambda \gg \alpha$ and $\alpha \Rightarrow_x \beta$, then $(\lambda + \alpha) \Rightarrow_x (\lambda + \beta)$;
    \item[(2)] If $k < \ell \in \NN$, then $\omega^\alpha \cdot \ell \Rightarrow_x \omega^\alpha \cdot k$;
    \item[(3)] If $\alpha \Rightarrow_x \beta$ and $\alpha > \beta+1$, then $\alpha \Rightarrow_{x+1} \beta+1$.
    \item[(4)] If $\alpha \Rightarrow_x \beta$ and $x \geq 1$, then $\omega^{\alpha} \Rightarrow_x \omega^{\beta}$;
    \item[(5)] If $x \geq 1$ and $k < \ell \in \NN$, then $\{\alpha\}(\ell) \Rightarrow_x \{\alpha\}(k)$.
\end{itemize}
\end{lemma}
\begin{proof}
(1) is \cite[Lemma 1, part 2.4]{ketonen1981rapidly},
(2) is \cite[Lemma 3, part 2.4]{ketonen1981rapidly},
(3) is \cite[Lemma 2, part 2.6]{ketonen1981rapidly},
(4) is \cite[Lemma 5, part 2.4]{ketonen1981rapidly},
(5) is \cite[Theorem 2.4]{ketonen1981rapidly} for $\alpha$ limit. If $\alpha$ is 0 or successor, then $\{\alpha\}(\ell) = \{\alpha\}(k)$.
\end{proof}

The following lemma states that the $\Rightarrow_x$ relation is a sufficient hypothesis for the Hardy-like hierarchy of functions to behave nicely. \Cref{lem:h-hierarchy}(1,3) is a Ketonen-Solovay counterpart of Bigorajska and Kotlarski~\cite[Lemma 4]{bigorajska1999partition}:

\begin{lemma}\label[lemma]{lem:h-hierarchy}
Let $h$ be as above. Let $0 < \alpha, \beta < \epsilon_0$ and $x \in \NN$. Then:
\begin{itemize}
    \item[(1)] $h_{\beta}$ is increasing 
    \item[(2)] if $x \in \dom h$ and $h(x) \in \dom h_{\beta}$ then $x \in \dom h_{\beta}$
    \item[(3)] if $\beta \Rightarrow_x \alpha$ and $x \in \dom h_{\beta}$ then $x \in \dom h_\alpha$ and $h_{\alpha}(x) \leq h_{\beta}(x)$
    \item[(4)] if $x \in \dom h_\beta$, then $h_\beta(x) \geq h(x)$
\end{itemize}
\end{lemma}
\begin{proof}
We prove (1),(2) and (3) by mutual induction on $\beta$. If $\beta = 1$ then the property holds as $h_{1} = h$ is increasing and the only $0<\alpha<\epsilon_0$ such that $\beta \Rightarrow_x \alpha$ is $\alpha = 1$. 

Let $\beta > 1$ and assume that $(1),(2)$ and $(3)$ hold for every $0<\beta' < \beta$. Let $x,y \in \dom h_{\beta}$ such that $x < y$, then $h_{\beta}(y) = h_{\{\beta\}(y)}(h(y))$. As $\{\beta \}(y) < \beta$ and as, by \Cref{lem:basic-arrow}(5), $\{\beta \}(y) \Rightarrow_{h(y)} \{\beta\}(x)$, then by $(3)$ for $\{\beta\}(y)$ we have $h(y) \in \dom h_{\{\beta\}(x)}$ and $h_{\{\beta\}(y)}(h(y)) \geq h_{\{\beta\}(x)}(h(y))$. From $x \in \dom h_{\beta}$ we get $h(x) \in \dom h_{\{\beta\}(x)}$, hence, by $(1)$ for $\{\beta\}(x)$, we get $h_{\{\beta\}(x)}(h(y)) > h_{\{\beta\}(x)}(h(x))$. Thus, $h_{\beta}(y) > h_{\beta}(x)$ and $h_{\beta}$ is increasing. Thus $(1)$ holds for $\beta$.

Let $x \in \dom h$ be such that $h(x) \in \dom h_{\beta}$. Then we have $h_{\beta}(h(x)) = h_{\{\beta\}(h(x))}(h(h(x)))$ hence $h(h(x)) \in \dom h_{\{\beta\}(h(x))}$. By $(2)$ for $\{\beta\}(h(x))$ we have $h(x) \in \dom h_{\{\beta\}(h(x))}$. Then, as \Cref{lem:basic-arrow}(5) $\{\beta\}(h(x)) \Rightarrow_{h(x)} \{\beta\}(x)$, we also have $h(x) \in \dom h_{\{\beta\}(x)}$ by $(3)$ for $\{\beta\}(h(x))$, hence $x \in \dom h_{\beta}$. Thus $(2)$ holds for $\beta$.

Let $x \in \NN$ and $0<\alpha$ be such that $\beta \Rightarrow_x \alpha$ and $h_{\beta}(x) \downarrow$. Since $\beta \Rightarrow_x \alpha$, either $\beta = \alpha$ or $\{\beta\}(x) \Rightarrow_x \alpha$. If $\beta = \alpha$ then $h_{\alpha}(x) \downarrow$ and $h_{\alpha}(x) = h_{\beta}(x)$, in which case (3) holds. If $\{\beta\}(x) \Rightarrow_x \alpha$, then as $h_{\beta}(x) = h_{\{\beta\}(x)}(h(x))$, we have $x \in \dom h_{\{\beta\}(x)}$ by $(2)$ for $\{\beta\}(x)$ and $h_{\{\beta\}(x)}(h(x)) > h_{\{\beta\}(x)}(x)$ by $(1)$ for $\{\beta\}(x)$. From $\{\beta\}(x) \Rightarrow_x \alpha$ and $(3)$ for $\{\beta\}(x)$, we get $x \in \dom h_{\alpha}$ and $ h_{\{\beta\}(x)}(x) \geq h_{\alpha}(x)$. So $h_{\beta}(x) = h_{\{\beta\}(x)}(h(x)) > h_{\{\beta\}(x)}(x) \geq h_{\alpha}(x)$, Thus $(3)$ holds for $\beta$.

Last, since $\beta \Rightarrow_x 1$ for every~$\beta > 0$ and $x \in \NN$, then (4) follows from (3).
\end{proof}

The following lemma is the Ketonen-Solovay counterpart of Bigorajska and Kotlarski~\cite[Lemma 7]{bigorajska1999partition}.
Formulated in terms of largeness, it corresponds to Ketonen and Solovay~\cite[Lemma 4.6]{ketonen1981rapidly}.

\begin{lemma}\label[lemma]{lem:h-composition}
    Let $\alpha \ll \beta$ and $x \in \NN$, if $h_{\beta + \alpha}(x) \downarrow$ then $h_{\beta}(h_{\alpha}(x)) \downarrow$ and $h_{\beta + \alpha}(x) = h_{\beta}(h_{\alpha}(x))$.
\end{lemma}
\begin{proof}
By induction on $\alpha$. If $\alpha = 0$ then, for $x \in \NN$ and $\beta \gg \alpha$ such that $h_{\beta + \alpha}(x) \downarrow$, we have $h_{\beta}(x) \downarrow$, then, as $h_{0}(x) = x$ we have $h_{\beta}(x) = h_{\beta}(h_{\alpha}(x)) \downarrow$ and $h_{\beta}(h_{\alpha}(x)) = h_{\beta + \alpha}(x)$ .

Let $\alpha > 0$ and assume the property to hold for every $\alpha' < \alpha$. Let $x \in \NN$ and $\beta \gg \alpha$, if $h_{\beta + \alpha}(x) \downarrow$ then $h_{\{\beta + \alpha\}(x)}(h(x)) \downarrow$ and $h_{\{\beta + \alpha\}(x)}(h(x)) = h_{\beta + \alpha}(x)$. By \Cref{lem:mesh-evaluation}, $\{\beta + \alpha\}(x) = \beta + \{\alpha\}(x)$, so $h_{\beta + \{\alpha\}(x)}(h(x)) = h_{\beta + \alpha}(x)$. As $\{\alpha\}(x) < \alpha$ and $\{\alpha\}(x) \ll \beta$, by the induction hypothesis we have $h_{\beta}(h_{\{\alpha\}(x)}(h(x))) \downarrow$ and $h_{\beta}(h_{\{\alpha\}(x)}(h(x))) = h_{\beta + \{\alpha\}(x)}(h(x)) = h_{\beta + \alpha}(x)$. And, as $h_{\{\alpha\}(x)}(h(x)) = h_{\alpha}(x)$, we have $h_{\beta + \alpha}(x) = h_{\beta}(h_{\alpha}(x))$.
\end{proof}

We now prove our core proposition. The intuition goes as follows: Suppose $0 < \lambda < \epsilon_0$ and $x > 0$.
The evaluation of $h_{\lambda}(x)$ unfolds into $h_{\{\lambda\}(x)}(h(x))$, then $h_{\{\{\lambda\}(x)\}(h(x))}(h(h(x))$, and so on.
Thus, the $i$th iteration of the evaluation of $h_\lambda(x)$, if defined, corresponds to $h_{\lambda_i}(x_i)$, where $x_i$ and $\lambda_i$ are defined inductively by $x_0 = x$, $\lambda_0 = \lambda$ and $x_{i+1} = h(x_i)$, $\lambda_{i+1} = \{\lambda_i\}(x_i)$. Note that since the function~$h$ is increasing, so is the sequence $x_0 < x_1 < \dots$

The definition of $\{\lambda\}(x)$ is done by considering the short Cantor normal form, by reducing first the smallest component of the Cantor normal form. The evaluation either decreases a successor ordinal by one, or replaces a limit ordinal by a smaller one, with coefficient~$x$. Since the sequence $\lambda_0 > \lambda_1 > \dots$ is defined by successive evaluations with an increasing sequence $x_0 < x_1 < \dots$, the higher components will be evaluated with larger values than the smaller components, and therefore this evaluation order ensures that the sequence $\lambda_0 > \lambda_1 > \cdots$ decreases as slowly as possible.

Based on this evaluation order, if $\lambda = \beta + \alpha$, with $\beta \gg \alpha$, the sequence $\lambda_0 > \lambda_1 > \cdots$ corresponds to evaluating $\alpha$ successively with $x_0, x_1, \dots$ until we reach a stage~$i$ such that $\alpha_i = 0$. Then we evaluate $\beta$ with $x_{i+1}, x_{i+2}, \dots$. 

One could consider other evaluation orders, in which one would alternate some evaluation steps of $\beta$, and some of $\alpha$. 
In terms of the Hardy-like hierarchy, one evaluation step of $\beta$ corresponds to $h_{\{\beta\}(x) \oplus \alpha}(h(x))$ while one evaluation step of $\alpha$ corresponds to $h_{\beta + \{\alpha\}(x)}(h(x))$. Since the standard evaluation order is the one making the sequence $\lambda_0 > \lambda_1 > \dots$ decrease as slowly as possible, this translates into the following inequality:
$$h_{\beta + \alpha}(x) = h_{\beta + \{\alpha\}(x)}(h(x)) \geq h_{\{\beta\}(x) \oplus \alpha}(h(x))$$

% One could consider another evaluation order, in which $\beta$ is first evaluated, until we reach a stage~$i$ such that $\beta_i = 0$, and then only we evaluate $\alpha$. In terms of the Hardy-like hierarchy, it corresponds to the computation of $h_{\{\beta\}(x) \oplus \alpha}(x)$ \quentin{Pas vraiment, ça correspondrait plutôt à $h_{\alpha}(h_{\beta}(x))$, $h_{\{\beta\}(x) \oplus \alpha}(x)$ correspond juste à un pas d'évaluation du coté $\beta$.}.
% Since the first evaluation order is the one making the sequence $\lambda_0 > \lambda_1 > \dots$ decrease as slowly as possible, this translates into the following inequality:
% $$h_{\beta + \alpha}(x) = h_{\beta + \{\alpha\}(x)}(h(x)) \geq h_{\{\beta\}(x) \oplus \alpha}(h(x))$$

The following proposition formalizes this intuition in the case $\beta \gg \alpha$. It will then be proven in its most general form in \Cref{thm:pigeonhole-one-step-not-ll}.

\begin{proposition}\label[proposition]{prop:pigeonhole-one-step}
    Let $0 < \alpha, \beta < \epsilon_0$ with $\beta \gg \alpha$ and $x > 0$ an integer. If $h_{\beta + \alpha}(x) \downarrow$, then $h_{\{\beta\}(x) \oplus \alpha}(h(x)) \downarrow$ and $h_{\beta + \alpha}(x) = h_{\beta + \{\alpha\}(x)}(h(x)) \geq h_{\{\beta\}(x) \oplus \alpha}(h(x))$.
\end{proposition}

\begin{proof}
The proof is by case analysis on the Cantor normal forms, and not by induction.
Write $\alpha = \omega^{\alpha_n} \cdot a_n + \dots + \omega^{\alpha_0} \cdot a_0$ and $\beta = \omega^{\beta_k} \cdot b_k + \dots + \omega^{\beta_0} \cdot b_0$ their Cantor normal form. From the assumption $\beta \gg \alpha$, we have $\beta_0 \geq \alpha_n$.

\textbf{Case 1: $n = 0$ and $\alpha_n = \beta_0$.} In that case, $\beta + \{\alpha\}(x) = \{\beta\}(x) \oplus \alpha$, hence $h_{\{\beta\}(x) \oplus \alpha}(h(x)) \downarrow$ and $h_{\beta + \{\alpha\}(x)}(h(x)) = h_{\{\beta\}(x) \oplus \alpha}(h(x))$.\\

If we are not in this case, we can freely assume that $\alpha_n < \beta_0$ in the Cantor normal form of $\alpha$ and $\beta$. Indeed, if $\alpha_n = \beta_0$, we can consider $\alpha' = \omega^{\alpha_{n-1}}a_{n-1} + \dots + \omega^{\alpha_0}a_0$ and $\beta' =  \omega^{\beta_k}b_k + \dots + \omega^{\beta_0}(b_0 + a_n)$, we then have $0 < \alpha', \beta' < \epsilon_0$, $\alpha' \ll \beta'$, $\beta + \alpha = \beta' + \alpha'$, $\beta + \{\alpha\}(x) = \beta' + \{\alpha'\}(x)$ and $\{\beta\}(x) \oplus \alpha = \{\beta'\}(x) \oplus \alpha'$, making it sufficient to prove the result for $\alpha'$ and $\beta'$.\\

%\textbf{Case 2: $\alpha = \gamma + \omega^{\alpha'}$ and $\beta = \delta + \omega^{\beta'}$ for some $\gamma,\delta < \epsilon_0$, some $\alpha' < \beta'$ with $\beta' = \beta'' + 1$ and some $0<a_0,b_0$}

\textbf{Case 2: $\alpha_n < \beta_0$ and $\beta_0$ is successor.} Let $\beta_0'$ be such that $\beta_0 = \beta_0' + 1$. In that case, write $\beta = \delta + \omega^{\beta_0'+1}$. By our assumption that $\alpha_n < \beta_0$, we have $\omega^{\beta_0'+1} > \alpha$ and $\omega^{\beta_0'} \gg \alpha$.
By \Cref{lem:h-composition}, letting $c = h_{\{\alpha\}(x)}(h(x)) = h_{\alpha}(x)$, we have 
$$h_{\beta + \{\alpha\}(x)}(h(x)) = h_{\beta}(c) = h_{\delta + \omega^{\beta'_0}\cdot c}(h(c))$$

Consider the following two subcases:

\textbf{Subcase 2.1: $\alpha < \omega$.} In that case, $\alpha = k$ for some $k \in \NN$. This gives us $h_{\beta + \alpha}(x) = h_{\beta}(h^k(x)) = h_{\{\beta\}(h^k(x))}(h^{k+1}(x))$ (where $h^{\ell}$ denotes the composition of $h$ with itself $\ell$ times). Then, as $h$ is increasing, we have $\{\beta\}(h^k(x)) \Rightarrow_{h^{k+1}(x)} \{\beta\}(x)$ by \Cref{lem:basic-arrow}(5). Hence, by \Cref{lem:h-hierarchy}(3), we have $h_{\{\beta\}(x)}(h^{k+1}(x)) \downarrow$ and 

$$h_{\{\beta\}(h^k(x))}(h^{k+1}(x)) \geq h_{\{\beta\}(x)}(h^{k+1}(x)) = h_{\{\beta\}(x) + \alpha}(h(x)) = h_{\{\beta\}(x) \oplus \alpha}(h(x))$$

Hence $h_{\beta + \alpha}(x) \geq h_{\{\beta\}(x) \oplus \alpha}(h(x))$, which is what we wanted.

\textbf{Subcase 2.2: $\alpha \geq \omega$.} In that case, we can use the strict inequality obtained in \Cref{lem:psn-nb-step} to get $c > \psn(\alpha) + x$.

As $\beta_0 > \alpha_n$, we get $\beta_0' \geq \alpha_n$ and therefore have $\omega^{\beta_0'} \cdot (\psn(\alpha) + 1) > \alpha$ by definition of $\psn(\alpha)$. Thus, by \Cref{lem:standard-sequence-vs-psn-2}, we have $\omega^{\beta_0'} \cdot (\psn(\alpha) + 1) \Rightarrow_{h(c)} \alpha$ (using the fact that $h(c) > c > \psn(\alpha)$).

We have $\omega^{\beta_0'}\cdot c \Rightarrow_{h(c)} \omega^{\beta_0'} \cdot (x + \psn(\alpha) + 1) = \omega^{\beta_0'} \cdot x + \omega^{\beta_0'} \cdot (\psn(\alpha) + 1)$ by \Cref{lem:basic-arrow}(2). Then, as $\omega^{\beta_0'} \cdot x \gg \omega^{\beta_0'} \cdot (\psn(\alpha) + 1)$, we can apply \Cref{lem:basic-arrow}(1) and obtain $\omega^{\beta_0'} \cdot x + \omega^{\beta_0'} \cdot (\psn(\alpha) + 1) \Rightarrow_{h(c)} \omega^{\beta_0'} \cdot x + \alpha$. Hence that $\omega^{\beta_0'}\cdot c \Rightarrow_{h(c)} \omega^{\beta_0'} \cdot x + \alpha$.

By another application of \Cref{lem:basic-arrow}(1) (using the fact that $\delta \gg \omega^{\beta_0'}\cdot c$), we obtain $\delta + \omega^{\beta_0'}\cdot c \Rightarrow_{h(c)} \delta + \omega^{\beta_0'} \cdot x + \alpha$.

By \Cref{lem:h-hierarchy}(3), we thus have $h_{\delta + \omega^{\beta_0'} \cdot x + \alpha}(h(c)) \downarrow$ and $$h_{\delta + \omega^{\beta_0'}\cdot c}(h(c)) \geq h_{\delta + \omega^{\beta_0'} \cdot x + \alpha}(h(c))$$

As $h(c) = h(h_{\alpha}(x)) =  h^{\ell}(h(x))$ for some $\ell \in \NN$, we have, by several applications of \Cref{lem:h-hierarchy}(2) that $h_{\delta + \omega^{\beta_0'} \cdot x + \alpha}(h(x)) \downarrow$ and that $$h_{\delta + \omega^{\beta_0'} \cdot x + \alpha}(h(c)) \geq h_{\delta + \omega^{\beta_0'} \cdot x + \alpha}(h(x)) = h_{\{\beta\}(x) \oplus \alpha}(h(x))$$

Which is what we wanted.

\textbf{Case 3: $\alpha_n < \beta_0$ and $\beta_0$ limit.} In that case, write $\beta = \delta + \omega^{\beta_0}$ and $\alpha = \gamma + \omega^{\alpha_0}$.
Similarly, by \Cref{lem:mesh-evaluation} and \Cref{lem:h-composition}, letting $c = h_{\{\alpha\}(x)}(h(x)) = h_{\alpha}(x)$, we have 
$$h_{\beta + \alpha}(x) = h_{\beta + \{\alpha\}(x)}(h(x)) = h_{\beta}(c) = h_{\delta + \omega^{\{\beta_0\}(c)}}(h(c))$$

Since $c > x$, we have $\{\beta_0\}(c) \Rightarrow_{y} \{\beta_0\}(x) + 1$ for every $y \geq 2$ by \Cref{lem:basic-arrow}(3,5). 

By \Cref{lem:psn-nb-step}, $c \geq \psn(\alpha) + x > \psn(\alpha)$. Thus, by \Cref{lem:standard-sequence-vs-psn}, we have $\{\beta_0\}(c) > \alpha_n$, and, by \Cref{lem:standard-sequence-vs-psn-2}, $\{\beta_0\}(c) \Rightarrow_c \alpha_n$. By \Cref{lem:basic-arrow}(3), we have $\{\beta_0\}(c) \Rightarrow_{h(c)} \alpha_n + 1$, as $h(c) > c$.

Thus, we have both $\{\beta_0\}(c) \Rightarrow_{h(c)} \{\beta_0\}(x) + 1$ and $\{\beta_0\}(c) \Rightarrow_{h(c)} \alpha_n + 1$. Depending on whether $\alpha_n$ or $\{\beta_0\}(x)$ appears first in the sequence $\{\beta_0\}(c) \Rightarrow_{h(c)} \{\{\beta_0\}(c)\}(h(c)) \Rightarrow_{h(c)} \{\{\{\beta_0\}(c)\}(h(c))\}(h(c)) \Rightarrow_{h(c)} \dots$, we either have $\{\beta_0\}(x) \Rightarrow_{h(c)} \alpha_n$ or $\alpha_n \Rightarrow_{h(c)} \{\beta_0\}(x)$. 

We can then consider the following two subcases:

\textbf{Subcase 3.1: $\alpha_n > \{\beta_0\}(x)$, in which case $\alpha_n \Rightarrow_{h(c)} \{\beta_0\}(x)$.} By \Cref{lem:basic-arrow}(2,4), we have $$\omega^{\{\beta_0\}(c)} \Rightarrow_{h(c)} \omega^{\alpha_n + 1} \Rightarrow_{h(c)} \omega^{\alpha_n}\cdot h(c) \Rightarrow_{h(c)} \omega^{\alpha_n}\cdot (c+1)$$

Thus, by \Cref{lem:basic-arrow}(1), since $\delta \gg \omega^{\{\beta_0\}(c)}$,

$$\delta + \omega^{\{\beta_0\}(c)} \Rightarrow_{h(c)} \delta +  \omega^{\alpha_n}\cdot (c+1)$$

Therefore, by \Cref{lem:h-hierarchy}(3), $h_{\delta + \omega^{\alpha_n}\cdot (c+1)}(h(c)) \downarrow$, and, letting $d = h_{\omega^{\alpha_n}}(h(c))$,

$$h_{\delta + \omega^{\{\beta_0\}(c)}}(h(c)) \geq h_{\delta + \omega^{\alpha_n}\cdot (c+1)}(h(c)) = h_{\delta + \omega^{\alpha_n}\cdot c}(d)$$

As $\omega^{\alpha_n} \Rightarrow_{h(c)} \omega^{\{\beta_0\}(x)}$ by \Cref{lem:basic-arrow}(4), we have $h_{\omega^{\{\beta_0\}(x)}}(h(c)) \downarrow$ and $d \geq h_{\omega^{\{\beta_0\}(x)}}(h(c))$ by \Cref{lem:h-hierarchy}(3).

By \Cref{lem:psn-nb-step}, $h_{\omega^{\{\beta_0\}(x)}}(h(c)) \geq \psn(\{\beta_0\}(x)) + h(c)$. Hence, we have $d \geq \psn(\{\beta_0\}(x)) + h(c) > \psn(\{\beta_0\}(x)) + \psn(\alpha)$ as $h(c) > c > \psn(\alpha)$.

Then, from $c > \psn(\alpha)$ and  $\alpha_n > \{\beta_0\}(x)$, we get $\omega^{\alpha_n}\cdot c > \omega^{\{\beta_0\}(x)} \oplus \alpha$, and, as $d > \psn(\omega^{\{\beta_0\}(x)}) + \psn(\alpha) \geq \psn(\omega^{\{\beta_0\}(x)} \oplus \alpha)$, we have $\omega^{\alpha_n}\cdot c \Rightarrow_{d} \omega^{\{\beta_0\}(x)} \oplus \alpha$ by \Cref{lem:standard-sequence-vs-psn-2}.

Again, as $\delta \gg \omega^{\{\beta_0\}(c)}$, we can apply \Cref{lem:basic-arrow}(1) and obtain $\delta + \omega^{\alpha_n}\cdot c \Rightarrow_{d} \delta + (\omega^{\{\beta_0\}(x)} \oplus \alpha) = \{\beta\}(x) \oplus \alpha$.

Thus, by \Cref{lem:h-hierarchy}(3), $h_{\{\beta\}(x) \oplus \alpha}(d) \downarrow$ and  

$$h_{\delta + \omega^{\alpha_n}\cdot c}(d) \geq h_{\{\beta\}(x) \oplus \alpha}(d)$$

As $d =  h^{\ell}(h(x))$ for some $\ell \in \NN$, we have, by several applications of \Cref{lem:h-hierarchy}(2) that $h_{\{\beta\}(x) \oplus \alpha}(h(x)) \downarrow$ and that $h_{\{\beta\}(x) \oplus \alpha}(d) \geq h_{\{\beta\}(x) \oplus \alpha}(h(x))$.

Combined with the other inequalities, we get $h_{\beta + \alpha}(x) \geq h_{\{\beta\}(x) \oplus \alpha}(h(x))$.

\textbf{Subcase 3.2: $\alpha_n \leq \{\beta_0\}(x)$, in which case $\{\beta_0\}(x) \Rightarrow_{h(c)} \alpha_n$.} By \Cref{lem:basic-arrow}(2,4), we have $$\omega^{\{\beta_0\}(c)} \Rightarrow_{h(c)} \omega^{\{\beta_0\}(x) + 1} \Rightarrow_{h(c)} \omega^{\{\beta_0\}(x)}\cdot h(c) \Rightarrow_{h(c)} \omega^{\{\beta_0\}(x)}\cdot (c+1)$$

Then, from $c > \psn(\alpha)$ and $\alpha_n \leq \{\beta_0\}(x)$, we get $\omega^{\{\beta_0\}(x)}\cdot c > \alpha$, hence, by \Cref{lem:standard-sequence-vs-psn-2} we get $\omega^{\{\beta_0\}(x)}\cdot c \Rightarrow_{h(c)} \alpha$. As $\omega^{\{\beta_0\}(x)} \gg \alpha$, \Cref{lem:basic-arrow}(1) gives us $$\omega^{\{\beta_0\}(x)}\cdot (c + 1) = \omega^{\{\beta_0\}(x)} + \omega^{\{\beta_0\}(x)}\cdot c \Rightarrow_{h(c)} \omega^{\{\beta_0\}(x)} + \alpha$$

Then, as $\delta \gg \omega^{\{\beta_0\}(c)}$, another application of \Cref{lem:basic-arrow}(1) gives us $$\delta + \omega^{\{\beta_0\}(c)} \Rightarrow_{h(c)} \delta + \omega^{\{\beta_0\}(x)} + \alpha = (\delta + \omega^{\{\beta_0\}(x)}) \oplus \alpha = \{\beta\}(x) \oplus \alpha$$

Thus, by \Cref{lem:h-hierarchy}(3), $h_{\{\beta\}(x) \oplus \alpha}(x) \downarrow$ and  

$$h_{\delta + \omega^{\{\beta_0\}(c)}}(h(c)) \geq h_{\{\beta\}(x) \oplus \alpha}(h(c))$$

Again, by several applications of \Cref{lem:h-hierarchy}(2), we have $h_{\{\beta\}(x) \oplus \alpha}(h(x)) \downarrow$ and $h_{\{\beta\}(x) \oplus \alpha}(h(c)) \geq h_{\{\beta\}(x) \oplus \alpha}(h(x))$.

Hence, we have $h_{\beta + \alpha}(x) \geq h_{\{\beta\}(x) \oplus \alpha}(h(x))$. \\

This completes our proof of \Cref{prop:pigeonhole-one-step}.
\end{proof}

\begin{lemma}\label[lemma]{lem:split-cantor-normal-form}
    Let $\alpha, \beta < \epsilon_0$, then there exists $\alpha', \beta' < \epsilon_0$ with $\alpha' \ll \beta'$ and such that $\beta \oplus \alpha = \beta' + \alpha'$ and $\{\beta\}(x) \oplus \alpha = \{\beta'\}(x) \oplus \alpha'$ for every $x \in \NN$
\end{lemma}

\begin{proof}
If $\beta = 0$ then take $\alpha' = \alpha$ and $\beta' = \beta$. If $\beta \neq 0$, then write $\alpha = \omega^{\gamma_n} \cdot a_n + \dots + \omega^{\gamma_0} \cdot a_0$ and $\beta = \omega^{\gamma_n} \cdot b_n + \dots + \omega^{\gamma_0} \cdot b_0$ with $a_0, \dots, a_n,b_0, \dots, b_n \in \NN$ and $\gamma_0 < \dots < \gamma_n < \epsilon_0$. Let $i \leq n$ be the smallest index such that $b_i \neq 0$, and let $\beta' = \omega^{\gamma_n} \cdot (b_n+a_n) + \dots + \omega^{\gamma_i} \cdot (b_i + a_i)$ and $\alpha' = \omega^{\gamma_{i-1}} \cdot (b_{i-1}+a_{i-1}) + \dots + \omega^{\gamma_0} \cdot (b_0 + a_0)$. 

We have 
$\beta \oplus \alpha = \beta' + \alpha' = \omega^{\gamma_n}(b_n+a_n) + \dots + \omega^{\gamma_0}(b_0 + a_0)$ and $\{\beta'\}(x) \oplus \alpha' = (\omega^{\gamma_n}(b_n+a_n) + \dots + \omega^{\gamma_i}(b_i + a_i - 1) + \dots + \omega^{\gamma_0}(b_0 + a_0)) \oplus \{\omega^{\gamma_i}\}(x) = \{\beta\}(x) \oplus \alpha$
\end{proof}

The following theorem informally states that the evaluation order of the fundamental sequences is the one which makes the Hardy-like hierarchy grow the fastest. In the case $\alpha \ll \beta$, $\beta + \alpha = \beta \oplus \alpha$, hence it generalizes \Cref{prop:pigeonhole-one-step}.

\begin{theorem}[Optimal growth]\label[theorem]{thm:pigeonhole-one-step-not-ll}
    Let $\alpha, \beta < \epsilon_0$ with $\beta > 0$ and $x  > 0$ an integer. If $h_{\beta \oplus \alpha}(x) \downarrow$ then $h_{\{\beta\}(x) \oplus \alpha}(h(x)) \downarrow$ and $h_{\beta \oplus \alpha}(x) = h_{\{\beta \oplus \alpha\}(x)}(h(x)) \geq h_{\{\beta\}(x) \oplus \alpha}(h(x))$.
\end{theorem}

\begin{proof}
Let $\alpha',\beta' < \epsilon_0$ be obtained from \Cref{lem:split-cantor-normal-form} applied on $\alpha, \beta$. As $\beta \oplus \alpha = \beta' + \alpha'$ and $\{\beta\}(x) \oplus \alpha = \{\beta'\}(x) \oplus \alpha'$, it is sufficient to show that $h_{\{\beta'\}(x) \oplus \alpha'}(h(x)) \downarrow$ and that $h_{\beta' \oplus \alpha'}(x) = h_{\{\beta' \oplus \alpha'\}(x)}(h(x)) \geq h_{\{\beta'\}(x) \oplus \alpha'}(h(x))$.
    
If $\alpha' = 0$, then the result holds by definition of $h_{\beta'}$.

If $\alpha' \neq 0$, then the result is that of \Cref{prop:pigeonhole-one-step}.
\end{proof}

\subsection{Partition theorem}

We now translate this abstract study of the Hardy-like hierarchy into the framework of $\alpha$-largeness, to prove \Cref{thm:epsilon0-pigeonhole}.

\begin{definition}
    Let $A = \{x_0, \dots, x_{n-1}\} \subseteq \NN$ a finite set. We define $h^A: A  \setminus \{\max A\} \to A$ by $h^A(x_i) = x_{i+1}$ for every $i < n-1$. In other words, $h^A$ is the successor operation on the ordered set~$(A, <_\NN)$.
\end{definition}

The following lemma is the heart of the correspondence between Hardy-like hierarchies and of Ketonen and Solovay's $\alpha$-largeness.

\begin{lemma}
Let $0 < \alpha < \epsilon_0$, then $A$ is at most $\alpha$-large iff $h_{\alpha}^{A}(\min A) \uparrow$.
\end{lemma}
\begin{proof}
By induction on $\alpha$. The result holds for $\alpha = 1$, as $h_{\alpha}^A(\min A) \uparrow$ if and only if $|A| \leq 1$.
Let $1 < \alpha < \epsilon_0$ and assume the property holds for every $\beta < \alpha$. Let $A = \{x_0, x_1, \dots,x_n\} \subseteq \NN$ be such that $h_{\alpha}^A(x_0) \uparrow$, then, $h_{\{\alpha\}(x_0)}^A(x_1) \uparrow$ and therefore $h_{\{\alpha\}(x_0)}^{A \setminus \{x_0\}}(x_1) \uparrow$. By the inductive hypothesis, $A \setminus \{x_0\}$ is therefore at most $\{\alpha\}(x_0)$-large, hence $A$ is at most $\alpha$-large.
\end{proof}

The following theorem is an iterated counterpart of \Cref{thm:pigeonhole-one-step-not-ll} in the language of 
$\alpha$-largeness. It has several consequences, such as the splitting property (\Cref{thm:splitting-property}) and multiple versions of the pigeonhole principle (\Cref{thm:pigeonhole-v1,thm:pigeonhole-v2})

\begin{repmaintheorem}{thm:epsilon0-pigeonhole}
    Let $\beta, \gamma < \epsilon_0$ be two ordinals, and $B,C \subseteq \NN \setminus \{0\}$ be at most $\beta$-large and $\gamma$-large respectively. Then $B \cup C$ is at most $(\beta \oplus \gamma)$-large.
\end{repmaintheorem}

\begin{proof}
By assumption we have $h^B_{\beta}(\min B) \uparrow$ and $h^C_{\gamma}(\min C) \uparrow$. We want to show that $h^{B\cup C}_{\beta \oplus \gamma}(\min \{\min B, \min C\}) \uparrow$.

Write $B \cup C = \{x_0 < x_1 < \dots < x_{k-1}\}$ and consider the sequences of ordinals $(\beta_i)_{i < k}$ and  $(\gamma_i)_{i < k}$ defined inductively as follows: $\beta_0 = \beta$ and $\gamma_0 = \gamma$ and if $\beta_i$ and $\gamma_i$ have been defined, let $\beta_{i+1} = \{\beta_i\}(x_i)$ and $\gamma_{i+1} = \gamma_i$ if $x_i \in B$ and $\beta_{i+1} = \beta_i$ and $\gamma_{i+1} = \{\gamma_i\}(x_i)$ if $x_i \in C$. 

Since $B$ is at most $\beta$-large, then $\{\beta\}(B \setminus \{\max B\}) > 0$, thus, if $x_i \in B$ then $\beta_i \neq 0$, similarly, if $x_i \in C$ then $\gamma_i \neq 0$. Hence, if $x_{k-1} \in B$ then $\beta_{k-1} \neq 0$ and if $x_{k-1} \in C$ then $\gamma_{k-1} \neq 0$, in every case $\beta_{k-1} \oplus \gamma_{k-1} \neq 0$.

If $h^{B \cup C}_{\beta_0 \oplus \gamma_0}(x_0) \downarrow$, then, by induction on $i < k$, using \Cref{thm:pigeonhole-one-step-not-ll}, we claim that for every $i < k$, we have $h^{B \cup C}_{\beta_i \oplus \gamma_i}(x_i) \downarrow$ and $h^{B \cup C}_{\beta_i \oplus \gamma_i}(x_i) \geq h^{B \cup C}_{\beta_{i+1} \oplus \gamma_{i+1}}(x_{i+1})$.
Thus, if $h^{B \cup C}_{\beta_0 \oplus \gamma_0}(x_0) \downarrow$ then $h^{B \cup C}_{\beta_{k-1} \oplus \gamma_{k-1}}(x_{k-1}) \downarrow$, which is only possible if  $\beta_{k-1} \oplus \gamma_{k-1} = 0$, contradiction. Thus $h^{B \cup C}_{\beta_0 \oplus \gamma_0}(x_0) \uparrow$ and $B \cup C$ is at most $(\beta \oplus \gamma)$-large.
\end{proof}

The following splitting property generalizes \cite[Lemma 4.6]{ketonen1981rapidly}.
Note that the reversal holds if $\beta \gg \alpha$.

\begin{theorem}[Splitting property]\label[theorem]{thm:splitting-property}
    Let $X \subseteq \NN \setminus \{0\}$ be $(\alpha \oplus \beta)$-large for $\alpha, \beta < \epsilon_0$, then there exists $X_0 < X_1 \subseteq X$ such that $X_0$ is $\alpha$-large and $X_1$ is $\beta$-large.
\end{theorem}

\begin{proof}
Let $X_0$ be the prefix of $X$ that is exactly $\alpha$-large, i.e., if $X = \{x_0, \dots, x_{k-1}\}$, we let $X_0 = \{x_0, \dots, x_{i}\}$ for $i < k$ the smallest index such that $\{\alpha\}(x_0, \dots, x_{i}) = 0$. Such a set $X_0$ exists, as every coefficient in the Cantor normal form of $\alpha \oplus \beta$ is bigger than the corresponding one in the Cantor normal form of $\alpha$.

Let $X_1 = X \setminus X_0$, $\star = \max X + 1$ and let $X' = X \cup \{\star\}$. By the contrapositive of \Cref{thm:epsilon0-pigeonhole}, since $X'= X_0 \cup (X_1 \cup \{\star\})$ is not at most $(\alpha \oplus \beta)$-large, either $X_0$ is not at most $\alpha$-large or $X_1 \cup \{\star\}$ is not at most $\beta$-large. Thus, $X_1 \cup \{\star\}$ is not at most $\beta$-large, and $X_1$ is therefore $\beta$-large.
\end{proof}

We now deduce two versions of the pigeonhole principle.

\begin{theorem}[First pigeonhole principle]\label[theorem]{thm:pigeonhole-v1}
    Let $X = X_0 \cup X_1 \subseteq \NN \setminus \{0\}$ be $(\alpha \oplus \beta)$-large for some $\alpha, \beta < \epsilon_0$, then either $X_0 \setminus \{\max X_0\}$ is $\alpha$-large or $X_1 \setminus \{\max X_1\}$ is $\beta$-large or $X_0$ and $X_1$ are exactly $\alpha$-large and exactly $\beta$-large, respectively. 
\end{theorem}

\begin{proof}
Suppose $X_0$ is at most $\alpha$-large, otherwise $X_0 \setminus \{\max X_0\}$ is $\alpha$-large and we are done.
Similarly, suppose $X_1$ is at most $\beta$-large, otherwise $X_1 \setminus \{\max X_1\}$ is $\beta$-large and we are done.
Let $\star = \max X + 1$. The set $X_0 \cup (X_1 \cup \{\star\})$ is not at most $(\alpha \oplus \beta)$-large, hence, by the contrapositive of \Cref{thm:epsilon0-pigeonhole}, $X_1 \cup \{\star\}$ is not at most $\beta$-large. Since $X_1$ is at most $\beta$-large, it is exactly $\beta$-large.
By applying the contrapositive of \Cref{thm:epsilon0-pigeonhole} on the partition $(X_0 \cup \{\star\}) \cup X_1$, we get that $X_0 \cup \{\star\}$ is not at most $\alpha$-large, hence since $X_0$ is at most $\alpha$-large, $X_0$ is exactly $\alpha$-large and we are also done since both $X_0$ and $X_1$ are exactly $\beta$-large.
\end{proof}

\begin{theorem}[Second pigeonhole principle]\label[theorem]{thm:pigeonhole-v2}
    Let $X = X_0 \cup X_1 \cup \{\star\} \subseteq \NN \setminus \{0\}$ be $(\alpha \oplus \beta)$-large for some $\alpha, \beta < \epsilon_0$ and some $\star > X_0 \cup X_1$. Then either $X_0$ is $\alpha$-large or $X_1$ is $\beta$-large.
\end{theorem}

\begin{proof}
Let $\star' = \max X + 1$. The set $X \cup \{\star'\} =  (X_0 \cup \{\star\}) \cup (X_1 \cup \{\star'\})$ is not at most $(\alpha \oplus \beta)$-large, hence, by the contrapositive of  \Cref{thm:epsilon0-pigeonhole}, either $X_0 \cup \{\star\}$ is not at most $\alpha$-large, in which case $X_0$ is $\alpha$-large, or $X_1 \cup \{\star'\}$ is not at most $\beta$-large, in which case $X_1$ is $\beta$-large.
\end{proof}

Since $\alpha$-largeness is closed under supersets, if $X_0 \cup X_1$ is $(\alpha \oplus \beta)$-large, then so is $X_0 \cup X_1 \cup \{\star\}$ for any $\star > \max X_1$, in which case either $X_0$ is $\alpha$-large or $X_1$ is $\beta$-large. This formulation of the pigeonhole principle, although more natural, is less optimal than \Cref{thm:pigeonhole-v2}.

\section{Largeness below $\omega^\omega$}\label[section]{sec:better-bounds-regular-largeness}

We now turn to the study of $\alpha$-largeness in the restricted setting of $\alpha < \omega^\omega$.
We shall give a particular focus to ordinals of the form $\omega^n \cdot k$ for $n, k \geq 1$. Thanks to Ketonen and Solovay~\cite[Lemma 4.6]{ketonen1981rapidly}, there exists a simple inductive definition of $\alpha$-largeness for ordinals of this form:

\begin{proposition}[Ketonen and Solovay~\cite{ketonen1981rapidly}]\label[proposition]{prop:largeness-below-omegaomega}
A set~$F \finsub \NN$ is 
\begin{itemize}
    \item[(1)] \emph{$\omega^0$-large} iff $F \neq \emptyset$
    \item[(2)] \emph{$\omega^{(n+1)}$-large} iff $F \setminus \min F$ is  $(\omega^n \cdot \min F)$-large
    \item[(3)] \emph{$\omega^n \cdot k$-large} iff
there are $k$ $\omega^n$-large subsets of~$F$
$$
F_0 < F_1 < \dots < F_{k-1}
$$
\end{itemize}
\end{proposition}
\begin{proof}
(1) and (2) are simply unfolding of the definition.
(3) is by induction on~$k$: The case $k = 1$ is trivial.  Assuming (3) holds for~$k$, by Ketonen and Solovay~\cite[Lemma 4.6]{ketonen1981rapidly}, a set~$X$ is $\omega^n \cdot (k+1)$-large iff there exist two sets~$X_0 < X_1 \subseteq X$ such that $X_0$ is $\omega^n$-large and $X_1$ is $\omega^n \cdot k$-large. Apply the induction hypothesis on~$X_1$.
\end{proof}

\subsection{Construction and deconstruction}

Thanks to \Cref{prop:largeness-below-omegaomega}, any $\omega^{n+1}$-large set~$X$ can be seen as an increasing sequence of 
$\min X$ many $\omega^n$-large blocks. In this section, we prove two important propositions: a \emph{construction} property, which quantifies how many $\omega^n$-large blocks are sufficient to produce an $\omega^{n+m}$-large set, and a \emph{deconstruction} property, which counts the number of $\omega^n$-large blocks which can be extracted from an $\omega^{n+m}$-large set.

The quantification is expressed in terms of $\alpha$-largeness. There are multiple ways to give a meaning to the sentence \qt{the sequence of blocks $X_0 < \dots < X_{d-1}$ is $\alpha$-large}. One could ask that $\{ \max X_i : i < d \}$ is $\alpha$-large, or that any sequence in $X_0 \times \dots \times X_{d-1}$ is $\alpha$-large. Thanks to the regularity of $\alpha$-largeness (\cite[Lemma 2]{ketonen1981rapidly}), the two definitions are equivalent.

\begin{proposition}[Construction]\label[proposition]{prop:regular-construction}
For every $n, m, d \in \NN$ and for every sequence of $\omega^n$-large sets $X_0 < \dots < X_{d-1}$ such that $\{\max X_i : i < d \}$ is $\omega^m$-large, then $X_0 \cup \dots \cup X_{d-1}$ is $\omega^{n+m}$-large.
\end{proposition}
\begin{proof}
By induction on~$m$. The case $m = 0$ is immediate, since any $\omega^0$-large set is non-empty, so $d \geq 1$ and $X_0$ is $\omega^n$-large.
Suppose it holds for~$m$. If $Y = \{ \max X_i : i < d \}$ is $\omega^{m+1}$-large, then $Y \setminus \{\min Y\}$ is $\omega^m \cdot \min Y$-large.
By \Cref{prop:largeness-below-omegaomega}, there are $\min Y$ many $\omega^m$-large subsets $Y_0 < \dots < Y_{\min Y - 1}$ of~$Y \setminus \{\min Y\}$.
For every~$i < \min Y$, let $Z_i = \bigcup_{j \in Y_i} X_i$. By induction hypothesis, each $Z_i$ is $\omega^{n+m}$-large, so $\{\min Y\} \cup Z_0 \cup \dots \cup Z_{\min Y-1}$ is $\omega^{n+m+1}$-large. By upward-closure of $\alpha$-largeness, $X_0 \cup \dots \cup X_{d-1}$ is $\omega^{n+m+1}$-large.
\end{proof}

The following deconstruction proposition is an adaptation of Ko{\l}odziejczyk and Yokoyama~\cite[Lemma 2.1]{kolo2020some}.
Note that the hypothesis requires an extra $\omega^n$-large set, while the conclusion also contains some extra $\omega^n$-large set. This might seem sub-optimal but the $\omega^n$-large set of the hypothesis is on the left part of~$X$, while the one from the conclusion is on the right, which is weaker.

\begin{proposition}[Deconstruction]\label[proposition]{prop:regular-deconstruction}
For every $n,m \in \NN$ and for every~$\omega^{n+m}+\omega^n$-large set~$X$, there are some $d \in \NN$ and some $\omega^n$-large subsets $X_0 < \dots < X_d$ of~$X$ such that $\{\max X_i : i < d\}$ is $\omega^m$-large.
\end{proposition}
\begin{proof}
By induction on~$m$. The case $m = 0$ is immediate by Ketonen and Solovay~\cite[Lemma 4.6]{ketonen1981rapidly}, as $X$ is $\omega^n \cdot 2$-large, so there are two $\omega^n$-large sets $X_0 < X_1 \subseteq X$. Note that $\{\max X_0\}$ is non-empty, hence $\omega^0$-large.

Suppose it holds for~$m$. Let $X$ be $\omega^{n+m+1} + \omega^n$-large. By Ketonen and Solovay~\cite[Lemma 4.6]{ketonen1981rapidly} (or by \Cref{thm:splitting-property}), there are some~$A < B \subseteq X$ such that $A$ is $\omega^n$-large and $B$ is $\omega^{n+m+1}$-large.
By \Cref{prop:largeness-below-omegaomega}, there are $\min B$ many $\omega^{n+m}$-large sets $Z_0 < \dots < Z_{\min B-1} \subseteq B \setminus \{\min B\}$. 

Let $d_{-1} = 0$ and $X^{-1}_{d_{-1}} = Z_0$. Note that $X^{-1}_{d_{-1}}$ is $\omega^{n+m}$-large, hence $\omega^n$-large.
We are going to build inductively for every~$i < \min B-1$ a family of $\omega^n$-large sets $X^i_0 < \dots < X^i_{d_i}$ such that
\begin{itemize}
    \item[(1)] $X^i_0 \cup \dots \cup X^i_{d_i} \subseteq X^{i-1}_{d_{i-1}} \cup Z_i$;
    \item[(2)] $\{ \max X^i_j : j < d_i \}$ is $\omega^m$-large.
\end{itemize}
Assuming $X^{i-1}_{d_{i-1}} < Z_i$ is defined, the set $X^{i-1}_{d_{i-1}} \cup Z_i$ is $\omega^{n+m} + \omega^n$-large, so by induction hypothesis, there is a sequence $X^i_0, \dots, X^i_{d_i}$ of $\omega^n$-large sets satisfying (1) and (2).
It follows that $\{ \max X^i_j : i < \min B-1, j < d_i \}$ is $\omega^m \cdot (\min B-1)$-large. Since $\max A < \min B$, then $\{\max A\} \cup \{ \max X^i_j : i < \min B-1, j < d_i\}$ is $\omega^{m+1}$-large, so altogether, these $\omega^n$-large sets form the desired sequence.
\end{proof}

% \begin{corollary}
% For every $n,m \in \NN$ and for every~$\omega^{n+m} \cdot 2$-large set~$X$ with $\min X \geq 2$, there are some $d \in \NN$ and some $\omega^n$-large subsets $X_0 < \dots < X_{d+m-1}$ of~$X$ such that $\{\max X_i : i < d\}$ is $\omega^m$-large.
% \end{corollary}
% \begin{proof}
% By \Cref{prop:regular-deconstruction}, it suffices to prove that every $\omega^{n+m} \cdot 2$-large set~$X$ with $\min X \geq 2$ is $\omega^{n+m} + \omega^n \cdot m$-large.
% For $m = 0$, any $\omega^n \cdot 2$-large set is $\omega^n$-large.
% For $m = 1$, any $\omega^{n+1} \cdot 2$-large set is $\omega^{n+1} + \omega^n$-large.

% For $m > 1$, if~$X$ is $\omega^{n+m} \cdot 2$-large, then it contains two $\omega^{n+m}$-large subsets~$X_0 < X_1 \subseteq X$.
% Then, $X_0 \setminus \{\min X_0\}$ is $\omega^{n+m-1} \cdot \min X_0$-large, hence $\omega^{n+m-1} \cdot 2$-large since $\min X \geq 2$.
% Let $Y_0 < Y_1 \subseteq X_0$ be two $\omega^{n+m-1}$-large subsets of~$X_0$. The set $Y_1$ is $\omega^{n+m-2} \cdot \min Y_1$-large, and in particular is $\omega^n \cdot m$-large, since $\min Y_1 \geq \max Y_0 \geq m$ and $m \geq 2$. Indeed, a simple induction shows that if $Y_0$ is $\omega^{n+m-1}$-large, then $\max Y_0 \geq m$.

% Since $Y_1 < X_1 \subseteq X$ with $Y_1$ $\omega^n \cdot m$-large and $X_1$ $\omega^{n+m}$-large, then by Ketonen and Solovay~\cite[Lemma 4.6]{ketonen1981rapidly}, $X$ is $\omega^{n+m} + \omega^n \cdot m$-large.
% \end{proof}

\begin{corollary}[Generalized deconstruction]\label[corollary]{cor:regular-deconstruction-improved}
For every~$n, k, \ell \geq 1$ and $m \geq 0$, for every~$\omega^{n+m} \cdot (k\ell+1)$-large set~$X$ such that $k(\ell+1) \leq \min X$,
there are some $d \in \NN$ and some $\omega^n \cdot k$-large subsets $X_0 < \dots < X_{d-1}$ of~$X$ such that $\{\max X_i : i < d\}$ is $\omega^m \cdot \ell$-large.
\end{corollary}
\begin{proof}
Suppose $m = 0$. Then $X$ is $\omega^n \cdot (k\ell+1)$-large, and a fortiori $\omega^n \cdot k\ell$-large hence there is a sequence $X_0 < \dots < X_{k\ell-1}$ of $\omega^n$-large subsets of~$X$.
For every~$i < \ell$, let $Z_i = X_{ki} \cup \dots \cup X_{k(i+1)-1}$. Each~$Z_i$ is $\omega^n \cdot k$-large, and $\{\max Z_i : i < \ell \}$ is $\ell$-large.

Suppose $m > 0$. In particular, $X$ is $\omega^{n+m} \cdot k\ell + \omega^{n+1}$-large.
Since $k(\ell+1) \leq \min X$, then $X$ is $\omega^{n+m} \cdot k\ell + \omega^n \cdot k(\ell+1)$-large.
By \Cref{thm:splitting-property}, there is a sequence $X_{-1} < X_0 < \cdots < X_{k\ell-1}$ of subsets of~$X$ such that $X_{-1}$ is $\omega^n \cdot k$-large and $X_0, \dots, X_{k\ell-1}$ are $\omega^{n+m}+\omega^n$-large. By \Cref{prop:regular-deconstruction}, for each~$i$ such that $0 \leq i < k\ell$, there is a sequence $X^0_i < \dots < X^{d_i-1}_i$ of $\omega^n$-large subsets of~$X_i$
such that $\{\max X^j_i : j < d_i \}$ is $\omega^m$-large. Let $W_0, \dots, W_{p-1}$ be the sequence $(X^j_i)_{i < k\ell, j < d_i}$.
In particular, the set $Z = \{ \max W_i : i < p \}$ is $\omega^m \cdot k\ell$-large. Let $g : Z \to k$ be such that $g(x)$ is the remainder of the euclidean division of~$i$ by $k$, where $x = \max W_i$.
By \Cref{thm:pigeonhole-v2}, there is an $\omega^m \cdot \ell$-large $g$-homogeneous subset~$H = \{ x_0 < \dots < x_{d-1} \} \subseteq Z$.
For each~$s < d$, let $i_s$ be such that $\max W_{i_s} = x_s$.
Let $H_0 = X_{-1} \cup W_{i_0}$ and $H_{s+1} = \bigcup_{i_s < j \leq i_{s+1}} W_j$.
Note that for each~$s < d$, $\max H_s = x_s$, so $\{\max H_s : s < d \}$ is $\omega^m \cdot \ell$-large.
Since $X_{-1}$ is $\omega^n \cdot k$-large, so is~$H_0$.
By $g$-homogeneity of~$H$, $i_{s+1} \geq i_s + k$, so $H_{s+1}$ is $\omega^n \cdot k$-large.
\end{proof}

\subsection{Sparsity}\label[section]{sec:sparsity}

The notion of $\alpha$-largeness induces a dual notion of $\alpha$-sparsity as follows:

\begin{definition}
A set $X \subseteq \NN$ is \emph{$\alpha$-sparse} if for every~$x, y \in X$ such that $x < y$,
$(x, y]_\NN$ is $\alpha$-large. 
\end{definition}

\begin{theorem}[Ketonen and Solovay~\cite{ketonen1981rapidly}]\label[theorem]{thm:large-dominates-pr}
For every primitive recursive function $g : \NN \to \NN$, there is some~$n \in \NN$
such that for every~$\omega^n$-large set~$F$, $\max F > g(\min F)$.
\end{theorem}

In particular, $y > 2x$ whenever $(x, y]$ is $\omega$-large and $y > x 2^x$ whenever $(x, y]$ is $\omega^2$-large.
We shall be particularly interested in $\omega^2 \cdot 3$-sparse sets as they provide sufficient sparsity to bound Ramsey numbers.

Let $R_k(d)$ be the least number~$R$ such that for every coloring $f : [R]^2 \to k$, there is an $f$-homogeneous set of size~$d$.
The standard upper bound for the finite Ramsey's theorem for pairs and two colors, obtained by Erd\H{o}s and Szekeres in \cite{erdos1935combinatorial}, is that $R_2(d) \leq \binom{2d-1}{d-1} \leq 4^{d}$. The upper bound generalizes to $R_k(d) \leq k^{kd}$ (see Graham, Rothschild and Spencer~\cite[Section 1.1]{graham2013ramsey}).

\begin{lemma}\label[lemma]{lem:omega3-3-sparsity}
If $X$ is $\omega^2 \cdot 3$-sparse and $\min X \geq 7$, then it is $(x \mapsto x^{R_x(2x+2)})$-sparse.
\end{lemma}
\begin{proof}
Any $\omega^2$-sparse set is $(x \mapsto x 2^x)$-sparse, so any $\omega^2 \cdot 3$-sparse set is $(x \mapsto x 2^{x + x 2^x +x 2^{x+x 2^x}})$-sparse, and in particular $(x \mapsto 2^{x 2^{x 2^x}})$-sparse. Whenever $x \geq 7$, then $x(2x+2) \leq 2^x$, so $2^{x 2^{x 2^x}} \geq x^{2^{x^2(2x+2)}} \geq x^{x^{x(2x+2)}}$. Thus $Y$ is $(x \mapsto x^{x^{x(2x+2)}})$-sparse.

Since $R_{x}(2x+2) \leq x^{x(2x+2)}$ (see Graham, Rothschild and Spencer~\cite[Section 1.1]{graham2013ramsey}), $x^{R_x(2x + 2)} \leq x^{x^{x(2x+2)}}$, so $X$ is $(x \mapsto x^{R_x(2x + 2)})$-sparse.
\end{proof}

Thanks to the deconstruction proposition \Cref{cor:regular-deconstruction-improved}, one can relate largeness to sparsity.
The following lemma specializes this relation in the case of $\omega^2 \cdot 3$-sparsity.

\begin{lemma}\label[lemma]{lem:sparsity-simplified}
For every $\ell \geq 1$ and $m \geq 0$, for every $\omega^{m+3}$-large set~$X$
such that $3\ell+2 \leq \min X$, there is some $\omega^2 \cdot 3$-sparse $\omega^m \cdot \ell+1$-large subset of~$X$.
\end{lemma}
\begin{proof}
Since $3\ell+2 \leq \min X$, then $X \setminus \{\min X\}$ is $\omega^{m+2} \cdot (3\ell+2)$-large, hence $\omega^{m+2} \cdot (3\ell+1) + \omega^{m+2}$-large. Then there are two subsets $A < B \subseteq X$ such that $A$ is $\omega^{m+2}$-large and $B$ is $\omega^{m+2} \cdot (3\ell+1)$-large.
Since $3\ell+2 \leq \min A$, then $3(\ell+1) \leq \max B$, so by 
\Cref{cor:regular-deconstruction-improved} applied to~$B$, there is a sequence  $X_0 < \dots < X_{d-1}$ of $\omega^2 \cdot 3$-large subsets of~$B$ such that $Z = \{ \max X_i : i < d \}$ is $\omega^m \cdot \ell$-large.
The set $\{\min X\} \cup Z$ is $\omega^m \cdot \ell+1$-large, and $\omega^2 \cdot 3$-sparse.
% \smallskip

% \textbf{Claim.} \emph{If $A$ is $\omega^m$-large and $\min A \geq x$, then $\max A \geq xm$.}
% If $x = 0$ or $x = 1$, then this is trivial. So suppose $x \geq 2$.
% By induction on~$m$. If $m = 0$, this is trivial as $\max A \geq 0$.
% Suppose it holds for~$m$. Suppose $A$ is $\omega^{m+1}$-large. The set $A' = A \setminus \{\min A\}$ is $\omega^m \cdot x$-large.
% There are $x$ $\omega^m$ large subsets $A_0  < \dots < A_{x-1} \subseteq A'$. The set $A_0$ satisfies $\min A_0 \geq (x+1)$,
% so by induction hypothesis, $\max A_0 \geq (x+1)m$. In particular, $\min A_1 \geq (x+1)m$, so $\max A_1 \geq (x+1)m^2$, and so on.
% So $\max A_{x-1} \geq (x+1)m^x$. Since $x \geq 2$, $(x+1)m^x \geq (x+1)m^2 \geq x(m+1)$, so $\max A \geq x(m+1)$. This completes the proof of the claim.
% \smallskip

% By the claim, since $A$ is $\omega^{m+2}$-large and $3\ell+2 \leq \min A$, then $\max A \geq (3\ell+2)(m+2) \geq 3(m\ell+1)$. In particular, $\min B \geq 3(m\ell+1)$, so by \Cref{cor:regular-deconstruction-improved} applied to~$B$, there is a sequence  $X_0 < \dots < X_{d-1}$ of $\omega^2 \cdot 3$-large subsets of~$B$ such that $Z = \{ \max X_i : i < d \}$ is $\omega^m \cdot \ell$-large.
% The set $\{\min X\} \cup Z$ is $\omega^m \cdot \ell+1$-large, and $\omega^2 \cdot 3$-sparse.
\end{proof}

\subsection{Largeness for combinatorial theorems}

We are particularly interested in the closure of notions of largeness under combinatorial statements coming from Ramsey theory.
Given a coloring $f : [\NN]^n \to k$, a set $H \subseteq \NN$ is \emph{$f$-homogeneous} if $f$ is constant on $[H]^n$.

\begin{statement}[Ramsey's theorem]
Given $n, k \in \NN$, $\RT^n_k$ is the statement \qt{For every coloring $f : [\NN]^n \to k$, there is an infinite $f$-homogeneous set}.
\end{statement}

Ramsey's theorem for pairs is a central subject of study in reverse mathematics, a foundational program whose goal is to find optimal axioms to prove ordinary theorems~\cite{dzhafarov2022reverse,simpson2009subsystems}. Its first-order consequences are closely related to the computation of bounds for its largeness counterpart:

\begin{definition}
A set~$X \finsub \NN$ is $\RT^n_k$-$\alpha$-large if for every coloring~$f : [X]^n \to k$, there is an $\alpha$-large $f$-homogeneous subset~$H \subseteq X$.
\end{definition}

Whenever a parameter is omitted, it is taken to be $\min X$. For instance, a set~$X$ is $\RT^n$-$\alpha$-large if it is $\RT^n_{\min X}$-$\alpha$-large,
in other words, if for every coloring $f : [X]^n \to \min X$, there is an $\alpha$-large $f$-homogeneous subset~$H \subseteq X$.
Bounds largeness for Ramsey's theorem for pairs were extensively studied. Ketonen and Solovay~\cite{ketonen1981rapidly} proved that every $\omega^{k+4}$-large set is $\RT^2_k$-$\omega$-large when $k \geq 2$. Patey and Yokoyama~\cite{patey2018proof} proved that for every~$k \in \NN$, there is some~$n \in \NN$ such that every $\omega^n$-large set is $\RT^2_2$-$\omega^n$-large. Ko{\l}odziejczyk and Yokoyama~\cite{kolo2020some} computed some explicit bounds and proved that every $\omega^{144(n+1)}$-large set is $\RT^2_2$-$\omega^n$-large. Our goal is to give a tight bound to $\RT^2_k$-$\omega^n$-largeness, up to an additive constant.

In order to better understand its computational and proof-theoretic strength, Bovykin and Weiermann~\cite{bovykin2017strength} decomposed Ramsey's theorem for pairs and two colors ($\RT^2_2$) into two statements: the Erd\H{o}s-Moser theorem ($\EM$) and the Ascending Descending Sequence ($\ADS$). 
The Erd\H{o}s-Moser theorem~\cite{erdos1964representation} is a statement from graph theory about tournaments, that is, complete oriented graphs, which says that any infinite tournaments admits an infinite transitive sub-tournament. The Ascending Descending Sequence principle~\cite{hirschfeldt2007combinatorial} is a theorem from order theory, stating that every infinite linear order admits an infinite ascending or descending sequence. Both statements can be formulated as particular cases of Ramsey's theorem for pairs, thanks to the notion of transitivity.

\begin{definition}\label[definition]{def:transitivity}
Given a coloring $f : [\NN]^2 \to k$, a set~$H \subseteq \NN$ is \emph{$f$-transitive} if for every
$x < y < z \in H$ such that $f(x, y) = f(y, z)$, then $f(x, y) = (x, z)$. Whenever~$H = \NN$, we simply say that the coloring $f$ is transitive.
\end{definition}

Tournaments are in one-to-one correspondence with 2-colorings of pairs, and transitive tournaments with transitive colorings.
Similarly, linear orders can be coded as transitive 2-colorings of pairs, in which case ascending or descending sequences correspond to homogeneous sets.
Therefore, given any instance of Ramsey's theorem for pairs and two colors, one can apply $\EM$ to obtain an infinite transitive subset, and then $\ADS$ to obtain a homogeneous subset. Both statements were extensively studied in computability theory and reverse mathematics~\cite{bovykin2017strength,lerman2013separating,hirschfeldt2007combinatorial,lerman1982recursive}.

When trying to extend the decomposition of $\RT^2_2$ into $\EM$ and $\ADS$ to more colors, there are two natural generalizations of the notion of transitivity: a weak one, given in \Cref{def:transitivity}, and a stronger one, formally introduced by Towsner and Yokoyama~\cite{towsner2024erdos}, but already used in Patey~\cite[Theorem 8.2.1]{patey2016reverse}.

\begin{definition}
Given a coloring $f : [\NN]^2 \to k$, a set $H \subseteq \NN$ is \emph{$f$-fallow} if for every~$x < y < z \in H$, $f(x, z) \in \{f(x, y), f(y, z)\}$.
\end{definition}

Clearly, any $f$-fallow set is $f$-transitive, but the converse is not true. Both $\EM$ and $\ADS$ therefore admit two generalized statements, based on transitivity and fallowness. We shall only consider the stronger version of each statement, that is, the fallow version of the Erd\H{o}s-Moser theorem and transitive Ramsey's theorem for pairs.

\begin{statement}[Fallow Erd\H{o}s-Moser theorem]
Given $k \in \NN$, $\fEM_k$ is the statement \qt{For every coloring $f : [\NN]^2 \to k$, there is an infinite $f$-fallow set}.
\end{statement}

\begin{statement}[Transitive Ramsey's theorem for pairs]
Given $k \in \NN$, $\trRT^2_k$ is the statement \qt{For every transitive coloring $f : [\NN]^2 \to k$, there is an infinite $f$-homogeneous set}.
\end{statement}

The statements $\fEM_k$ and $\trRT^2_k$ admit largeness counterparts, as Ramsey's theorem. In the remainder of this article, we shall compute bounds for $\fEM$-$\omega^n$-largeness and $\trRT^2_k$-$\omega^n$-largeness, and deduce bounds for $\RT^2_k$-$\omega^n$-largeness.

\subsection{Pigeonhole principle}

We already studied $\alpha$-large versions of the pigeonhole principle for $\alpha < \epsilon_0$ in \Cref{sec:largeness-epsilon0}.
The following proposition is an immediate consequence of \Cref{thm:pigeonhole-v2}, restricted to the setting of $\omega^n$-largeness.
It is however useful to state it under this form as it will be used all over the remainder of the article. This strengthens Ko{\l}odziejczyk and Yokoyama~\cite[Lemma 2.2]{kolo2020some} by improving the bounds and removing the sparsity assumptions.

\begin{proposition}\label[proposition]{prop:pigeonhole-principle-enhanced}
Fix~$n \geq 0$ and $a, k \geq 1$.
\begin{itemize}
    \item[(1)] If $X$ is $\omega^n \cdot ak$-large, then it is $\RT^1_k$-$\omega^n \cdot a$-large.
    \item[(2)] If $X$ is $\omega^{n+1}$-large, then it is $\RT^1$-$\omega^n$-large.
\end{itemize}
\end{proposition}
\begin{proof}
(1) By induction on~$k$. The case~$k = 1$ is trivial. Suppose it holds for~$k$. Let $X$ be $\omega^n \cdot a(k+1)$-large, and let $f : X \to k+1$. Let $Y_0 = \{ x \in X : f(x) < k \}$ and $Y_1 = \{ x \in X : f(x) = k \}$. By \Cref{thm:pigeonhole-v2}, either $Y_0$ is $\omega^n \cdot ak$-large, in which case by induction hypothesis, there is an $f$-homogeneous set which is $\omega^n \cdot a$-large,
or $Y_1$ is $\omega^n \cdot a$-large, and is $f$-homogeneous for color~$k$ by definition.

(2) Let $f : X \to \min X$. Since $X$ is $\omega^{n+1}$-large, $X \setminus \{\min X\}$ is $\omega^n \cdot \min X$-large.
By (1), $X \setminus \{\min X\}$ is $\RT^1_{\min X}$-$\omega^n$-large, so there is an $\omega^n$-large $f$-homogeneous subset.
\end{proof}

\subsection{Grouping principle}

%The upper bound generalizes to $R_k(d) \leq k^{(d-1)k+1}-1$ (see Graham, Rothschild and Spencer~\cite[Section 1.1]{graham2013ramsey}).

The inductive proofs of Ko{\l}odziejczyk and Yokoyama for largeness for~$\EM$ heavily depended on the notion of grouping. We shall only use the following lemma, which, in the terminology of grouping, means that every $\omega^n \cdot R_2(2d)$-large $(x \mapsto x2^x)$-sparse set admits an $(\omega^n+1, d)$-grouping. This is an improvement of their bound, which used an $\omega^{n+3}$-large set \cite[Lemma 2.6]{kolo2020some}. 

\begin{lemma}\label[lemma]{lem:new-stabilize-bounds-optimal}
    Let $k > 0$ and let $X \subseteq_{fin} \NN$ be $\omega^{n} \cdot R_k(2d)$-large and $(x \mapsto xk^x)$-sparse, for some $d \in \NN$ such that $k^{R_k(2d)} \leq \min X$ and let $f : [X]^2 \to k$ be a coloring. 
    Then, there exists a sequence $Y_0 < \dots < Y_{d-1}$ of $\omega^n+1$-large subsets of $X$ and some color $c < k$ such that $f(x,y) = c$ for every $x \in Y_s$ and $y \in Y_t$ for $s < t < d$.
\end{lemma}

\begin{proof}
If $n = 0$, the results holds by definition of $R$ and $\omega^0$-largeness.
Indeed, let $f : [X]^2 \to k$ be a coloring. Since $|X| \geq R_k(2d)$, by finite Ramsey's theorem for pairs, there is an $f$-homogeneous set $\{y_0 < y_1 < \dots < y_{2d-1} \} \subseteq X$. Since any singleton element is $\omega^0$-large, the set $Y_i = \{y_{2i}, y_{2i+1}\}$ is $\omega^0+1$-large. Assume from now on that $n > 0$.

Let $X_0 < \dots < X_{R_k(2d)-1}$ be a decomposition of $X$ into $\omega^{n}$-large sets. 
Consider the sequence $(f_i)_{i < R_k(2d)}$ of colorings and the sequence $Z_0 < \dots < Z_{R_k(2d)-1}$ of sets defined inductively as follows:

Assume that $Z_{i+1}, \dots, Z_{R_k(2d) - 1}$ have been defined for some $i < R_k(2d)$, and consider the coloring $f_i : X_i \to k^{R_k(2d) - 1 - i + \sum_{j < i} |X_j|}$ defined by 
$$f_i(x) = (f(y,x))_{y \in \bigcup_{j < i} X_i} \sqcup (f(x,y))_{y \in \{\min Z_j : i < j < R_k(2d)\}}$$ 
As $\min X_0 \geq R_k(2d)$ and as the $X_i$ are disjoint, we get that for $i > 0$, $R_k(2d) + \sum_{j < i} |X_j| \leq \max X_{i-1}$. By $(x \mapsto xk^x)$-sparsity of~$X$, $k^{\max X_{i-1}} \cdot \max X_{i-1} \leq \min X_i$, so $k^{R_k(2d) + \sum_{j < i} |X_j|} \cdot \max X_{i-1} \leq \min X_i$. By \Cref{prop:pigeonhole-principle-enhanced}(1), there exists some $\omega^{n-1}\cdot (\max X_{i-1})$-large and $f_i$-homogeneous subset $Z_{i} \subseteq X_{i}$. For $i = 0$, since $k^{R_k(2d)} \leq \min X_0$, by \Cref{prop:pigeonhole-principle-enhanced}(1), there exists some $\omega^{n-1}$-large and $f_0$-homogeneous subset $Z_{0} \subseteq X_{0}$.

The $(Z_i)_{i < R_k(2d)}$ are defined such that, for every $i < j < R_k(2d)$ and $x \in Z_i$, $y \in Z_j$, we have $f(x,y) = f(\min Z_i, \min Z_j)$. By definition of $R_k(2d)$, there exists some subset $\{i_0 < \dots < i_{2d - 1}\} \subseteq \{0, \dots, R_k(2d)-1\}$ and some color $c < k$ such that, for every $j < k < 2d$, $x \in Z_{i_j}$ and $y \in Z_{i_k}$, $f(x,y) = c$.

Finally, let $Y_j = \{z^0_{2j}, z^1_{2j} \} \cup Z_{i_{2j+1}}$ for every $j < d$, where $z^0_{2j}$ and $z^1_{2j}$ are the two last elements of $Z_{i_{2j}}$. Every $Y_j$ is $\omega^{n}+1$-large.% hence the set $Y = \bigcup_{j < d} Y_j$ is $\omega^n \cdot d$-large.
\end{proof}

\subsection{Erd\H{o}s-Moser theorem}

The goal of this section is to prove that every $\omega^{n+3}$-large set~$X$ with $\min X \geq 5$ is $\EM$-$\omega^n$-large. The previous known bound of Ko{\l}odziejczyk and Yokoyama~\cite{kolo2020some} was $\omega^{36n+4}$-largeness.

\begin{proposition}\label[proposition]{prop:em-k-large}
Let $n \geq 0$, let $X \finsub \NN$ be and $(x \mapsto x^{R_x(2x)})$-sparse, %and let $d \geq 1$ be such that $2^{R_2(2d)} \leq \min X$.
\begin{itemize}
    \item[(1)] If $X$ is $\omega^n+1$-large with $2 \leq \min X$, then it is $\fEM$-$\omega^n$-large.
    \item[(2)] If $X$ is $\omega^n \cdot R_k(2d)$-large with $k^{R_k(2d)} \leq \min X$, then it is $\fEM_k$-$\omega^n \cdot d$-large.
    \item[(3)] If $X$ is $\omega^n \cdot 4$-large, then it is $\fEM$-$\omega^n+1$-large.
\end{itemize}
\end{proposition}
\begin{proof}
We prove (1) and (2) by mutual induction on~$n$.
(1) holds trivially for $n = 0$, as any $\omega^0$-large set is non-empty, and any singleton is $\EM$-$\omega^0$-large.
\bigskip

Suppose (1) holds for~$n$. Let us show that $(2)$ holds for~$n$. Let $k,d \in \NN$, let $X \finsub \NN$ be $(x \mapsto x^{R_x(2x)})$-sparse, $(\omega^{n} \cdot R_k(2d))$-large and such that $k^{R_k(2d)} \leq \min X$ and fix some coloring $f : [X]^2 \to k$. As $k^{R_k(2d)} \leq \min X$, we have $k \leq \min X$, hence $X$ is $(x \mapsto xk^x)$-sparse, and by \Cref{lem:new-stabilize-bounds-optimal}, there exists some sequence $Z_0 < \dots < Z_{d-1}$ of $\omega^n+1$-large subsets of~$X$ and some color~$c < k$ such that $f(x,y) = c$ for every $x \in Z_s$ and $y \in Z_t$ for $s < t < d$. By (1) for~$n$, for each~$i < d$, there is an $f$-fallow $\omega^n$-large subset $W_i \subseteq Z_i$. Then, the set $W = \bigcup_{i < d} W_i$ is $\omega^n \cdot d$-large and $f$-fallow. 

%Indeed, pick by contradiction some $x < y < z \in W$ such that $f(x,y) = f(y,z)$ and $f(x,y) \neq f(x,z)$. We cannot have $Z_i \cap \{x,y,z\} = \{x\}$ or $Z_i \cap \{x,y,z\} = \{z\}$ for some $i < d$ as it would imply $f(x,y) = f(x,z) = c$ or $f(x,z) = f(y,z) = c$ by our application of \Cref{lem:new-stabilize-bounds-optimal}. Similarly we cannot have $Z_i \cap \{x,y,z\} = \{y\}$ for some $i < d$ as it implies the existence of some $j < d$ such that $Z_j \cap \{x,y,z\} = \{z\}$, since $x < y < z$. And finally, we cannot have $Z_i \cap \{x,y,z\} = \{x,y,z\}$ for some $i < d$ by our application of the inductive hypothesis.
Indeed, pick by contradiction some $x < y < z \in W$ such that $f(x,z) \notin \{f(x,y), f(y,z)\}$. We cannot have $Z_i \cap \{x,y,z\} = \{x\}$ or $Z_i \cap \{x,y,z\} = \{z\}$ for some $i < d$ as it would imply $f(x,y) = f(x,z) = c$ or $f(x,z) = f(y,z) = c$ by our application of \Cref{lem:new-stabilize-bounds-optimal}. Similarly we cannot have $Z_i \cap \{x,y,z\} = \{y\}$ for some $i < d$ as it implies the existence of some $j < d$ such that $Z_j \cap \{x,y,z\} = \{z\}$, since $x < y < z$. And finally, we cannot have $Z_i \cap \{x,y,z\} = \{x,y,z\}$ for some $i < d$ by our application of the inductive hypothesis.
\bigskip

Suppose (2) holds for $n$. Let us show (1) holds for~$n+1$. Let $X \finsub \NN$ be $(x \mapsto x^{R_x(2x)})$-sparse and $(\omega^{n+1} + 1)$-large and let $x_0 = \min X$. Fix some coloring $f : [X]^2 \to x_0$, then, $X \setminus \{x_0\}$ is $\omega^{n} \cdot 2R_{x_0}(2x_0)$-large as $2R_{x_0}(2x_0) \leq x_0^{R_{x_0}(2x_0)}$ when $x_0 \geq 2$ and $x_0^{R_{x_0}(2x_0)} \leq \min X \setminus \{x_0\}$. By \Cref{prop:pigeonhole-principle-enhanced}(1), there exists some $\omega^{n} \cdot R_{x_0}(2x_0)$-large subset $Y \subseteq X \setminus \{x_0\}$ such that $f(x_0, y) = f(x_0,z)$ for every $y,z \in Y$.
By (2) for $n$, there is some $f$-fallow $\omega^n \cdot x_0$-large subset~$Z \subseteq Y$.
Then, the set $W = \{x_0\} \cup Z$ is $f$-fallow and $\omega^{n+1}$-large.

Indeed, pick by contradiction some $x < y < z \in W$ such that $f(x,z) \notin \{f(x,y), f(y,z)\}$.
We cannot have $x = x_0$ as $f(x_0,y) = f(x_0, z)$ for every $y,z \in Y$. And we cannot have $x_0 \in Z$, otherwise $\{x, y, z \} \subseteq Z$, contradicting the fact that $Z$ is $f$-fallow.
\bigskip

(3) For $n = 0$, if $X$ is $\omega^0 \cdot 4$-large, then $\card X \geq 4$. In particular, any two-element subset of~$X$ is $\fEM$-$\omega^0 +1$-large. Suppose $n > 0$. Fix a coloring $f : [\NN]^2 \to \min X$. Let $x_0 = \min X$ and $g : X \to x_0$ be defined by $g(x_0) = 0$ and $g(y) = f(x_0, y)$ for $y \neq x_0$. By \Cref{prop:pigeonhole-principle-enhanced}(1), there is an $\omega^n \cdot 2$-large $g$-homogeneous subset~$Y \subseteq X$. Let $Y_0 < Y_1$ be its decomposition into $\omega^n$-large sets. Let $x_1 = \max Y_0$.
Since $n > 0$, $\card Y_0 \geq 2$, so $x_1 > x_0$ and $\{x_1\} \cup Y_1$ is $\omega^n+1$-large. By (1), there is a $\omega^n$-large $f$-fallow subset~$Z \subseteq \{x_1\} \cup Y_1$. 
Since $Z$ is $g$-homogeneous, then $\{x_0\} \cup Z$ is $f$-fallow, and $\omega^n+1$-large.
\end{proof}

\begin{corollary}\label[corollary]{cor:em-n+3-large}
Let $n \geq 0$. If $X \finsub \NN$ is $\omega^{n+3}$-large and $\min X \geq 7$, then it is $\fEM$-$\omega^n$-large.
\end{corollary}
\begin{proof}
Since $\min X \geq 5$, then by \Cref{lem:sparsity-simplified}, there is an $\omega^2 \cdot 3$-sparse $\omega^n+1$-large subset~$Y \subseteq X$.
By \Cref{lem:omega3-3-sparsity}, since $\min X \geq 7$, $Y$ is $(x \mapsto x^{R_x(2x+2)})$-sparse, so by \Cref{prop:em-k-large}, $Y$ is $\fEM$-$\omega^n$-large.
\end{proof}

\subsection{Transitive colorings}

The strategy to compute bounds for transitive Ramsey's theorem for pairs and $k$-colors is different from the one for fallow Erd\H{o}s-Moser theorem.
Patey and Yokoyama~\cite[Lemma 4.4]{patey2018proof} proved that every $\omega^{2n+6}$-large set is $\ADS$-$\omega^n$-large, using a slightly modified notion of largeness. The proof consisted in defining a coloring $f : [\NN]^2 \to 2n+2$ and applying Ketonen and Solovay's bound~\cite{ketonen1981rapidly}. Ko{\l}odziejczyk and Yokoyama~\cite[Theorem 2.11]{kolo2020some} gave a more direct but less optimal proof by getting rid of the intermediate notion of largeness, and showed that any $\omega^{4n+4}$-large set~$X$ with $\min X \geq 3$ is $\trRT^2_2$-$\omega^n$-large.

In this section, we give a direct and optimal proof, by generalizing the bound of Ketonen and Solovay to $\omega$-$g$-largeness.
Given an increasing function $g : \NN \to \NN$, a set $X$ is \emph{$\omega$-$g$-large} if $\card X > g(\min X)$. Whenever $g$ is the identity function, this yields the standard notion of $\omega$-largeness. The following proposition is a generalization of Ketonen and Solovay~\cite[Lemma 6.3]{ketonen1981rapidly} and of Ko{\l}odziejczyk and Yokoyama~\cite[Theorem 1.5]{kolo2020some}:

\begin{proposition}\label[proposition]{prop:ketonen-solovay-revisited}
Fix a primitive recursive function~$g : \NN \to \NN$.
If $X$ is $(\omega^{k} + 1)$-large and $g$-sparse for $k = 1$ and $(x \mapsto k^{R_k(2g(x))})$-sparse for $k \geq 2$, then it is $\RT^2_k$-$\omega$-$g$-large.
\end{proposition}

\begin{proof}
By induction on $k$.
Suppose $k = 1$. Let $f : X \to 1$ be the trivial constant coloring. Then $X$ is $f$-homogeneous. Since $X$ is $\omega^k+1$-large, then, letting $X = \{x_0 < x_1\} \sqcup Y$, then $\card Y \geq x_1$ and by $g$-sparsity of~$X$, $x_1 > g(x_0)$, so $\card X > g(\min X)$.

Assume the property to be true for some $k > 0$. Let $X$ be $(\omega^{k+1} + 1)$-large and $(x \mapsto (k+1)^{R_{k+1}(2g(x))})$-sparse and consider a coloring $f : [X]^2 \to k+1$. Let $x_0 = \min X$, by sparsity, $\min (X \setminus \{x_0\}) \geq (k+1)^{R_{k+1}(2g(x_0))}$ and $X \setminus \{x_0\}$ is an $(\omega^k\cdot (k+1) \times R_{k+1}(2g(x_0)))$-large set.

Consider the coloring $h : X \setminus \{x_0\} \to k+1$ defined by $h(y) = f(x_0,y)$. By \Cref{prop:pigeonhole-principle-enhanced}(1), there exists some $(\omega^k\cdot R_{k+1}(2g(x_0)))$-large and $h$-homogeneous subset $Y$ of $X \setminus \{x_0\}$, and let $c_0 \leq k$ be the corresponding color.

By \Cref{lem:new-stabilize-bounds-optimal}, as $\min Y \geq (k+1)^{R_{k+1}(2g(x_0))}$, there exists a sequence $Y_0 < \dots < Y_{g(x_0) - 1}$ of $\omega^k+1$-large subsets of $Y$ and some color $c_1$ such that $f(x,y) = c_1$ for every $x \in Y_s$ and $y \in Y_t$ for $s < t < g(x_0)$.

If $c_0 = c_1$, then the set $\{x_0\} \cup \{\min Y_0 < \dots < \min Y_{g(x_0) - 1}\}$ is $f$-homogeneous and $\omega$-$g$-large and we are done. So assume $c_0 \neq c_1$ and consider the coloring $f' : [X]^2 \to k$ to be the same coloring as $f$, except that the colors $c_0$ and $c_1$ are fuse into one. By the inductive hypothesis, there exists some $f'$-homogeneous subsets $Z_{1} \subseteq Y_{1}$. 

If $Z_1$ is $f$ homogeneous, then we are done. Otherwise, $f([Z_1]^2) = \{c_0, c_1\}$, and, since $\min Z_{1} > R_2(g(\min Z_0))$ by sparsity of $X$, there exists some $f$-homogeneous subset $W \subseteq Z_{1}$ of size $\geq g(\min Z_0)$. If $f([W]^2) = \{c_0\}$, then $\{x_0\} \cup W$ is $f$-homogeneous, and if $f([W]^2) = \{c_1\}$, then $\{\min Z_0\} \cup W$ is $f$-homogeneous, and both of these sets are $\omega$-$g$-large.  
\end{proof}

By generalizing the proof of Patey and Yokoyama~\cite[Lemma 4.4]{patey2018proof} to $\trRT^2_k$ and replacing the bound of Ketonen and Solovay by \Cref{prop:ketonen-solovay-revisited}, we obtain the following proposition:

\begin{proposition}\label[proposition]{prop:trRT-nk-large}
Let $n \in \NN$, $k \geq 1$ and let $X \subseteq \NN$ be $(\omega^{kn}+1)$-large and $(x \mapsto (kn)^{R_{kn}(2x + 2)})$-sparse, then it is $\trRT^2_k$-$\omega^n$-large. 
\end{proposition}

\begin{proof}
Consider a transitive coloring $f : [X]^2 \to k$ and let $\bar{f} : [X]^2 \to kn$ be defined by $\bar{f}(x,y) = kj+i$ if $f(x,y) = i$ and if $j$ is the smallest index such that there is no $\omega^{j+1}$-large subset $H \subseteq [x,y) \cap X$ with $x \in H$ and such that $H\cup\{y\}$ is $f$-homogeneous for color~$i$, or, if no such index exists, take $j = n-1$.

By \Cref{prop:ketonen-solovay-revisited}, take $Y \subseteq X$ such that $Y$ is $\omega$-$(x \mapsto x+1)$-large and $\bar{f}$-homogeneous. Write $Y = \{y_0 < y_1 < \dots < y_{\ell}\}$ for some $\ell > y_0$ and let $\bar{f}([Y]^2) = kj+i$. For every $s < \ell$, let $H_s \subseteq [y_s, y_{s+1})$ be $\omega^j$-large, $f$-homogeneous for the color $i$ and such that $y_s \in H_s$. Let $H = \{y_0\} \cup \bigcup_{s = 1}^{\ell -1} H_s$, then, as $\ell > y_0$, $H$ is $\omega^{j+1}$-large and $H \cup \{y_{\ell}\}$ is $f$-homogeneous for the color $i$. Thus, we have $j = n$ as $\bar{f}(y_0, y_{\ell}) = nj + i$, hence $H$ is $f$-homogeneous and $\omega^n$-large.
\end{proof}

\begin{corollary}\label[corollary]{cor:trRT-explicit}
Let $n \in \NN$, $k \geq 1$. If $X \finsub \NN$ is $\omega^{nk+3}$-large and $\min X \geq 8$, then it is $\trRT^2_k$-$\omega^n$-large.
\end{corollary}
\begin{proof}
Since $\min X \geq 8$, then by \Cref{lem:sparsity-simplified}, there is an $\omega^2 \cdot 3$-sparse $\omega^{nk}\cdot 2+1$-large subset~$Y \subseteq X$.
By regularity, there are two subsets $Y_0 < Y_1$ of $Y$ such that $Y_0$ is $\omega^{nk}$-large and $Y_1$ is $\omega^{nk}+1$-large.  

By \Cref{lem:omega3-3-sparsity}, since $Y_1$ is $\omega^2 \cdot 3$-sparse and $\min X \geq 7$, then it is $(x \mapsto x^{x^{x(2x+2)}})$-sparse.
Since $Y_0$ is $\omega^{nk}$-large, then $\card Y_0 \geq nk$, so $\min Y_1 \geq nk$. It follows that $Y_1$ is $(x \mapsto (kn)^{(kn)^{kn(2x+2)}})$-sparse, so by \Cref{prop:trRT-nk-large}, $Y_1$ is $\trRT^2_k$-$\omega^n$-large.
\end{proof}

\begin{corollary}\label[corollary]{cor:ads-2n+3-large}
Let $n \in \NN$. If $X \finsub \NN$ is $\omega^{2n+3}$-large and $\min X \geq 9$, then it is $\ADS$-$\omega^n$-large.
\end{corollary}
\begin{proof}
Immediate by \Cref{cor:trRT-explicit}.
\end{proof}

Note that this bound is optimal, up to an additive constant, in that the coloring witnessing the lower bound of $\RT^2_2$ by Kotlarski et al.~\cite[Theorem 5.4]{kotlarski2007more} is transitive.

\subsection{Ramsey's theorem for pairs}

We now have all the necessary ingredients to prove our optimal bounds for $\RT^2_k$-$\omega^n$-largeness,
using the bounds for the fallow Erd\H{o}s-Moser theorem and for transitive Ramsey's theorem for pairs.
We first state it with largeness and sparsity assumptions, and then use \Cref{sec:sparsity} to compute purely largeness bounds.

\begin{theorem}\label[theorem]{thm:rt22-largeness-revisited}
Let $n, k \geq 1$. If $X \finsub \NN$ is $\omega^{kn} \cdot 4$-large and $(x \mapsto x^{R_{x}(2x + 2)})$-sparse with $kn \leq \min X$, then it is $\RT^2_k$-$\omega^n$-large.
\end{theorem}
\begin{proof}
Let $f : [X]^2 \to k$ be a coloring. $X$ is $\omega^{kn} \cdot 4$-large and $(x \mapsto x^{R_{x}(2x)})$-sparse, hence, by \Cref{prop:em-k-large}(3), there exists some $\omega^{kn} + 1$-large and $f$-fallow subset $Y \subseteq X \setminus \{\min X\}$. In particular, $Y$ is $f$-transitive.

As $kn \leq \min X$, $Y$ is $(x \mapsto (kn)^{R_{kn}(2x + 2)})$-sparse, and by \Cref{prop:trRT-nk-large} there exists some $\omega^n$-large $f$-homogeneous subset $Z \subseteq Y$.
\end{proof}

Note that if we applied $\fEM$ and $\trRT^2_k$ under their explicit largeness assumptions (\Cref{cor:em-n+3-large,cor:trRT-explicit}) instead of their hybrid sparsity/largeness versions (\Cref{prop:em-k-large,prop:trRT-nk-large}), this would have yielded significantly worse bounds. It is therefore useful to keep a sparsity version of the statements in case of their use as an intermediate object in the proof of a stronger combinatorial statement.
When translating the sparsity assumptions into largeness ones, we obtain our main theorem.

\begin{repmaintheorem}{thm:rt22-tight-bound}
Let $n, k \geq 1$. If $X \finsub \NN$ is $\omega^{kn+3}$-large and $\min X \geq 17$, then it is $\RT^2_k$-$\omega^n$-large. 
\end{repmaintheorem}
\begin{proof}
Since $\min X \geq 17$, then by \Cref{lem:sparsity-simplified}, there is an $\omega^2 \cdot 3$-sparse $\omega^{nk}\cdot 5+1$-large subset~$Y \subseteq X$.
In particular, there are two subsets $X_0 < X_1$ of $X$ such that $X_0$ is $\omega^{nk}$-large and $X_1$ is $\omega^{nk} \cdot 4$-large.
Since $X_0$ is $\omega^{nk}$-large, then $\card X_0 \geq nk$, so $\min X_1 \geq nk$.
Since $X_1$ is $\omega^2 \cdot 3$-large, then by \Cref{lem:omega3-3-sparsity}, it is 
$(x \mapsto x^{R_{x}(2x + 2)})$-sparse.
It follows by \Cref{thm:rt22-largeness-revisited} that $X_1$ is $\RT^2_k$-$\omega^n$-large.
\end{proof}

\bibliographystyle{plain}
\bibliography{biblio}

@article {ketonen1981rapidly,
    AUTHOR = {Ketonen, Jussi and Solovay, Robert},
     TITLE = {Rapidly growing {R}amsey functions},
   JOURNAL = {Ann. of Math. (2)},
  FJOURNAL = {Annals of Mathematics. Second Series},
    VOLUME = {113},
      YEAR = {1981},
    NUMBER = {2},
     PAGES = {267--314},
      ISSN = {0003-486X},
   MRCLASS = {03F30 (03C62 03D20 05A17)},
  MRNUMBER = {607894},
MRREVIEWER = {Leonard Lipshitz},
       DOI = {10.2307/2006985},
       URL = {https://doi.org/10.2307/2006985},
}

@article {patey2018proof,
    AUTHOR = {Patey, Ludovic and Yokoyama, Keita},
     TITLE = {The proof-theoretic strength of {R}amsey's theorem for pairs
              and two colors},
   JOURNAL = {Adv. Math.},
  FJOURNAL = {Advances in Mathematics},
    VOLUME = {330},
      YEAR = {2018},
     PAGES = {1034--1070},
      ISSN = {0001-8708},
   MRCLASS = {03B30 (03C62 03D80 03F35 03H15 05D10)},
  MRNUMBER = {3787563},
MRREVIEWER = {Fran\c{c}ois G. Dorais},
       DOI = {10.1016/j.aim.2018.03.035},
       URL = {https://doi.org/10.1016/j.aim.2018.03.035},
}

@article {kolo2020some,
    AUTHOR = {Ko{\l}odziejczyk, Leszek Aleksander and Yokoyama, Keita},
     TITLE = {Some upper bounds on ordinal-valued {R}amsey numbers for
              colourings of pairs},
   JOURNAL = {Selecta Math. (N.S.)},
  FJOURNAL = {Selecta Mathematica. New Series},
    VOLUME = {26},
      YEAR = {2020},
    NUMBER = {4},
     PAGES = {Paper No. 56, 18},
      ISSN = {1022-1824},
   MRCLASS = {03F30 (03B30 03F35 05D10)},
  MRNUMBER = {4125988},
MRREVIEWER = {Lars Halln\"{a}s},
       DOI = {10.1007/s00029-020-00577-3},
       URL = {https://doi.org/10.1007/s00029-020-00577-3},
}

@article{bovykin2017strength,
  title={The strength of infinitary Ramseyan principles can be accessed by their densities},
  author={Bovykin, Andrey and Weiermann, Andreas},
  journal={Annals of Pure and Applied Logic},
  volume={168},
  number={9},
  pages={1700--1709},
  year={2017},
  publisher={Elsevier}
}

@article{kolodziejczyk2023ramsey,
  title={Ramsey’s theorem for pairs, collection, and proof size},
  author={Ko{\l}odziejczyk, Leszek Aleksander and Wong, Tin Lok and Yokoyama, Keita},
  journal={Journal of Mathematical Logic},
  pages={2350007},
  year={2023},
  publisher={World Scientific}
}

@article {bigorajska1999partition,
    AUTHOR = {Bigorajska, Teresa and Kotlarski, Henryk},
     TITLE = {A partition theorem for {$\alpha$}-large sets},
   JOURNAL = {Fund. Math.},
  FJOURNAL = {Fundamenta Mathematicae},
    VOLUME = {160},
      YEAR = {1999},
    NUMBER = {1},
     PAGES = {27--37},
      ISSN = {0016-2736,1730-6329},
   MRCLASS = {05A18 (03F15)},
  MRNUMBER = {1694401},
       DOI = {10.4064/fm-160-1-27-37},
       URL = {https://doi.org/10.4064/fm-160-1-27-37},
}

@article{houerou2023conservation,
  title={{$\Pi^0_4$} conservation of the {O}rdered {V}ariable {W}ord theorem},
  author={Le Hou{\'e}rou, Quentin and Patey, Ludovic Levy},
  journal={The Journal of Symbolic Logic},
  pages={1--16},
  year={2023},
  publisher={Cambridge University Press}
}

@article {erdos1935combinatorial,
    AUTHOR = {Erd\"{o}s, P. and Szekeres, G.},
     TITLE = {A combinatorial problem in geometry},
   JOURNAL = {Compositio Math.},
  FJOURNAL = {Compositio Mathematica},
    VOLUME = {2},
      YEAR = {1935},
     PAGES = {463--470},
      ISSN = {0010-437X,1570-5846},
   MRCLASS = {DML},
  MRNUMBER = {1556929},
       URL = {http://www.numdam.org/item?id=CM_1935__2__463_0},
}

@book {graham2013ramsey,
    AUTHOR = {Graham, Ronald L. and Rothschild, Bruce L. and Spencer, Joel
              H.},
     TITLE = {Ramsey theory},
    SERIES = {Wiley Series in Discrete Mathematics and Optimization},
      NOTE = {Paperback edition of the second (1990) edition [MR1044995]},
 PUBLISHER = {John Wiley \& Sons, Inc., Hoboken, NJ},
      YEAR = {2013},
     PAGES = {xiv+196},
      ISBN = {978-1-118-79966-6},
   MRCLASS = {05-02 (05A99 05C55 54H20)},
  MRNUMBER = {3288500},
}

@article {bigorajska2006some,
    AUTHOR = {Bigorajska, Teresa and Kotlarski, Henryk},
     TITLE = {Partitioning {$\alpha$}-large sets: some lower bounds},
   JOURNAL = {Trans. Amer. Math. Soc.},
  FJOURNAL = {Transactions of the American Mathematical Society},
    VOLUME = {358},
      YEAR = {2006},
    NUMBER = {11},
     PAGES = {4981--5001},
      ISSN = {0002-9947,1088-6850},
   MRCLASS = {05A18 (03F15 03F30 05D10)},
  MRNUMBER = {2231881},
MRREVIEWER = {Martin\ Klazar},
       DOI = {10.1090/S0002-9947-06-03883-9},
       URL = {https://doi.org/10.1090/S0002-9947-06-03883-9},
}

@book {kotlarski2019model,
    AUTHOR = {Kotlarski, Henryk},
     TITLE = {A model-theoretic approach to proof theory},
    SERIES = {Trends in Logic---Studia Logica Library},
    VOLUME = {51},
      NOTE = {Edited by Zofia Adamowicz, Teresa Bigorajska and Konrad
              Zdanowski},
 PUBLISHER = {Springer, Cham},
      YEAR = {2019},
     PAGES = {xviii+109},
      ISBN = {978-3-030-28920-1; 978-3-030-28921-8},
   MRCLASS = {03-02 (03Cxx 03F30 03F40)},
  MRNUMBER = {3969945},
MRREVIEWER = {Roman\ Kossak},
}

@article {ratajczyk1988combinatorial,
    AUTHOR = {Ratajczyk, Z.},
     TITLE = {A combinatorial analysis of functions provably recursive in
              {${\rm I}\Sigma_n$}},
   JOURNAL = {Fund. Math.},
  FJOURNAL = {Polska Akademia Nauk. Fundamenta Mathematicae},
    VOLUME = {130},
      YEAR = {1988},
    NUMBER = {3},
     PAGES = {191--213},
      ISSN = {0016-2736,1730-6329},
   MRCLASS = {03F30 (03D20 03F10 03F15)},
  MRNUMBER = {970904},
MRREVIEWER = {A.\ M.\ Dawes},
       DOI = {10.4064/fm-130-3-191-213},
       URL = {https://doi.org/10.4064/fm-130-3-191-213},
}

@article{hardy1904theorem,
  title={A theorem concerning the infinite cardinal numbers},
  author={Hardy, Godfrey H},
  journal={Quarterly journal of Mathematics},
  volume={35},
  pages={87--94},
  year={1904}
}

@article {wainer1970classification,
    AUTHOR = {Wainer, S. S.},
     TITLE = {A classification of the ordinal recursive functions},
   JOURNAL = {Arch. Math. Logik Grundlag.},
  FJOURNAL = {Archiv f\"{u}r Mathematische Logik und Grundlagenforschung},
    VOLUME = {13},
      YEAR = {1970},
     PAGES = {136--153},
      ISSN = {0003-9268},
   MRCLASS = {02F35},
  MRNUMBER = {294134},
MRREVIEWER = {A.\ Kino},
       DOI = {10.1007/BF01973619},
       URL = {https://doi.org/10.1007/BF01973619},
}

@article {kotlarski2007more,
    AUTHOR = {Kotlarski, Henryk and Piekart, Bo\.{z}ena and Weiermann,
              Andreas},
     TITLE = {More on lower bounds for partitioning {$\alpha$}-large sets},
   JOURNAL = {Ann. Pure Appl. Logic},
  FJOURNAL = {Annals of Pure and Applied Logic},
    VOLUME = {147},
      YEAR = {2007},
    NUMBER = {3},
     PAGES = {113--126},
      ISSN = {0168-0072,1873-2461},
   MRCLASS = {03F30 (03E02 05A18 05D10)},
  MRNUMBER = {2335083},
MRREVIEWER = {G.\ E.\ Mints},
       DOI = {10.1016/j.apal.2006.04.004},
       URL = {https://doi.org/10.1016/j.apal.2006.04.004},
}

@article {bigorajska2002combinatorics,
    AUTHOR = {Bigorajska, Teresa and Kotlarski, Henryk},
     TITLE = {Some combinatorics involving {$\xi$}-large sets},
   JOURNAL = {Fund. Math.},
  FJOURNAL = {Fundamenta Mathematicae},
    VOLUME = {175},
      YEAR = {2002},
    NUMBER = {2},
     PAGES = {119--125},
      ISSN = {0016-2736,1730-6329},
   MRCLASS = {05D10 (05A18)},
  MRNUMBER = {1969630},
       DOI = {10.4064/fm175-2-2},
       URL = {https://doi.org/10.4064/fm175-2-2},
}

@article{smet2010partitioning,
  title={Partitioning $\alpha $-large sets for $\alpha<\varepsilon_\omega$},
  author={De Smet, Michiel and Weiermann, Andreas},
  journal={arXiv preprint arXiv:1001.2437},
  year={2010}
}

@unpublished{pelupessy2016largeness,
Author = {Florian Pelupessy},
Title = {On $\alpha$-largeness and the Paris-Harrington principle in $\mathrm{RCA}_0$ and $\mathrm{RCA}_0^{\displaystyle{*}}$},
Year = {2016},
Eprint = {arXiv:1611.08988},
}

@article {towsner2024erdos,
    AUTHOR = {Towsner, Henry and Yokoyama, Keita},
     TITLE = {Erd{\H{o}}s-Moser and {$I\Sigma_2$}},
   JOURNAL = {Israel J. Math.},
  FJOURNAL = {Israel Journal of Mathematics},
    VOLUME = {263},
      YEAR = {2024},
    NUMBER = {2},
     PAGES = {843--870},
      ISSN = {0021-2172,1565-8511},
   MRCLASS = {03F35 (03B30)},
  MRNUMBER = {4819968},
       DOI = {10.1007/s11856-024-2643-8},
       URL = {https://doi.org/10.1007/s11856-024-2643-8},
}

@article{houerou2026conservationramsey,
author = {Le Houérou, Quentin and Levy Patey, Ludovic and Yokoyama, Keita},
title = {$\Pi^0_4$ conservation of Ramsey's theorem for pairs},
journal = {Journal of the London Mathematical Society},
volume = {113},
number = {1},
pages = {e70419},
doi = {https://doi.org/10.1112/jlms.70419},
year = {2026}
}

@book {dzhafarov2022reverse,
    AUTHOR = {Dzhafarov, Damir D. and Mummert, Carl},
     TITLE = {Reverse mathematics---problems, reductions, and proofs},
    SERIES = {Theory and Applications of Computability},
 PUBLISHER = {Springer, Cham},
      YEAR = {[2022] \copyright 2022},
     PAGES = {xix+488},
      ISBN = {978-3-031-11366-6; 978-3-031-11367-3},
   MRCLASS = {03-02 (03B30 03F35)},
  MRNUMBER = {4472209},
MRREVIEWER = {Huishan Wu},
       DOI = {10.1007/978-3-031-11367-3},
       URL = {https://doi.org/10.1007/978-3-031-11367-3},
}

@book {simpson2009subsystems,
    AUTHOR = {Simpson, Stephen G.},
     TITLE = {Subsystems of second order arithmetic},
    SERIES = {Perspectives in Logic},
   EDITION = {Second},
 PUBLISHER = {Cambridge University Press, Cambridge; Association for
              Symbolic Logic, Poughkeepsie, NY},
      YEAR = {2009},
     PAGES = {xvi+444},
      ISBN = {978-0-521-88439-6},
   MRCLASS = {03F35 (03-02 03B30)},
  MRNUMBER = {2517689},
       DOI = {10.1017/CBO9780511581007},
       URL = {https://doi.org/10.1017/CBO9780511581007},
}

@article {erdos1964representation,
    AUTHOR = {Erd\H{o}s, P. and Moser, L.},
     TITLE = {On the representation of directed graphs as unions of
              orderings},
   JOURNAL = {Magyar Tud. Akad. Mat. Kutat\'{o} Int. K\"{o}zl.},
  FJOURNAL = {A Magyar Tudom\'{a}nyos Akad\'{e}mia. Matematikai Kutat\'{o}
              Int\'{e}zet\'{e}nek K\"{o}zlem\'{e}nyei},
    VOLUME = {9},
      YEAR = {1964},
     PAGES = {125--132},
      ISSN = {0541-9514},
   MRCLASS = {05.60},
  MRNUMBER = {168494},
MRREVIEWER = {J.\ W.\ Moon},
}

@article {lerman2013separating,
    AUTHOR = {Lerman, Manuel and Solomon, Reed and Towsner, Henry},
     TITLE = {Separating principles below {R}amsey's theorem for pairs},
   JOURNAL = {J. Math. Log.},
  FJOURNAL = {Journal of Mathematical Logic},
    VOLUME = {13},
      YEAR = {2013},
    NUMBER = {2},
     PAGES = {1350007, 44},
      ISSN = {0219-0613,1793-6691},
   MRCLASS = {03F35 (03B30)},
  MRNUMBER = {3125903},
MRREVIEWER = {Alberto\ Marcone},
       DOI = {10.1142/S0219061313500074},
       URL = {https://doi.org/10.1142/S0219061313500074},
}

@article {hirschfeldt2007combinatorial,
    AUTHOR = {Hirschfeldt, Denis R. and Shore, Richard A.},
     TITLE = {Combinatorial principles weaker than {R}amsey's theorem for
              pairs},
   JOURNAL = {J. Symbolic Logic},
  FJOURNAL = {The Journal of Symbolic Logic},
    VOLUME = {72},
      YEAR = {2007},
    NUMBER = {1},
     PAGES = {171--206},
      ISSN = {0022-4812,1943-5886},
   MRCLASS = {03F35 (03B30)},
  MRNUMBER = {2298478},
MRREVIEWER = {Roman\ Kossak},
       DOI = {10.2178/jsl/1174668391},
       URL = {https://doi.org/10.2178/jsl/1174668391},
}

@incollection {lerman1982recursive,
    AUTHOR = {Lerman, Manuel and Rosenstein, Joseph G.},
     TITLE = {Recursive linear orderings},
 BOOKTITLE = {Patras {L}ogic {S}ymposion ({P}atras, 1980)},
    SERIES = {Stud. Logic Found. Math.},
    VOLUME = {109},
     PAGES = {123--136},
 PUBLISHER = {North-Holland, Amsterdam-New York},
      YEAR = {1982},
      ISBN = {0-444-86476-8},
   MRCLASS = {03D45},
  MRNUMBER = {694256},
MRREVIEWER = {J.\ C.\ Owings, Jr.},
}

@phdthesis{patey2016reverse,
  title={The reverse mathematics of Ramsey-type theorems},
  author={Patey, Ludovic},
  year={2016},
  school={Universit{\'e} Paris Diderot (Paris 7) Sorbonne Paris Cit{\'e}}
}

\end{document}